\documentclass[11pt,a4paper]{article}

\usepackage[english]{babel}
\usepackage{array}
\usepackage{float}
\usepackage{cite}
\usepackage{booktabs} % for much better looking tables
\usepackage{subfig} % make it possible to include more than one captioned figure/table in a single float
\usepackage{amsfonts,amssymb,dsfont}
\usepackage{amsmath}
\usepackage{color}
\usepackage{natbib}
\usepackage{amsmath}
\usepackage{amsfonts}
\usepackage{amssymb}
\usepackage{amsthm}
\usepackage{dsfont}
\usepackage{graphicx}
\usepackage{enumitem}

\theoremstyle{plain}
\newtheorem{remark}{Remark}[section]
\newtheorem{theorem}{Theorem}[section]
\newtheorem{lemma}[theorem]{Lemma}
\newtheorem{proposition}[theorem]{Proposition}
\newtheorem{corollary}[theorem]{Corollary}

\begin{document}
% "Title of the Paper"
\title{Posterior consistency for nonparametric hidden Markov models with finite state space}

\author{Elodie Vernet \\ elodie.vernet@math.u-psud.fr }%\ead[label=e1]{elodie.vernet@math.u-psud.fr}}
%\address{Laboratoire de Math\'ematiques, Universit\'e Paris-Sud, Orsay, France\\  \printead{e1}}

\maketitle

\begin{abstract}
In this paper we study posterior consistency for different topologies on the parameters for hidden Markov models
with finite state space.
We first obtain weak and strong posterior consistency for the marginal density function of finitely many consecutive observations.
We deduce posterior consistency for the different components of the parameter. 
\textcolor{black}{ We also obtain posterior consistency for marginal smoothing distributions in the discrete case.}
We finally apply our results to independent emission probabilities, translated emission probabilities and discrete HMMs,
under various types of priors. 

\vspace{0.5cm}
\end{abstract}
%
%\begin{keyword}[class=MSC]
%\kwd[Primary ]{62G20}
%\end{keyword}
%

\section{Introduction}
\label{Sect:intro}

Hidden Markov models (HMMs) have been widely used in diverse fields such as speech recognition, 
genomics, econometrics since their introduction in \citet{BaPe66}. The books \citet{macdonald:zucchini:1997},  
\citet{macdonald:zucchini:2009} and \citet{CaMoRy05} provide several examples of applications of HMMs and give
a recent (for the latter) state of the art in the statistical analysis of HMMs. 
Finite state space HMMs are stochastic processes $(X_t,Y_t)_{t\in \mathbb{N}}$ such that $(X_t)_{t\in \mathbb{N}}$
is a Markov chain taking values in a finite set, and 
conditionally to $(X_t)_{t\in \mathbb{N}}$, the random variables $Y_t$, $t\in \mathbb{N}$, are independent,
the distribution of $Y_{t}$ depending only on $X_{t}$. The conditional distributions of $Y_{t}$ given $X_{t}$
for all possible values of $X_t$ are called emission distributions. 
The name ``hidden Markov model'' comes from the fact that the observations are 
 the $Y_{t}$'s only, 
one cannot access to the states $(X_t)_t$ of the Markov chain.
Finite state space HMMs can be used to model heterogeneous variables coming from different populations, the states of the (hidden) Markov chain defining the population the observed variable comes from.
%Given a state $X_t=i$ (a population), the observation $Y_t$ is distributed from a probability $p_i$ 
%and the next state is distributed from a probability vector $Q_{i,\cdot}$.
%$(Q_{i,j})_{i,j}$ is called a transition matrix and specifies the transition
%from one population to another when $t$ increases.
%The probabilities $(p_i)_i$ are the emission probabilities.
HMMs are very popular dynamical models especially because of their computational tractability
since there exist efficient algorithms to compute the likelihood and to recover the posterior distribution
of the hidden states given the observations.

Frequentist asymptotic properties of estimators of HMMs parameters have been studied since the 1990s.
Consistency and asymptotic normality of the maximum likelihood estimator have been established in the parametric case, see \citet{DoMa01}, \citet{Do04} and references in \citet{CaMoRy05},
see also \citet{DoMoOlHa11} for the most general consistency result up to now. 
As to Bayesian asymptotic results, there are only very few and recent results,
see \citet{GuSh08} when the number of hidden states is known,  \citet{GaRo12} when the number of hidden states is unknown.
All these results concern parametric HMMs.

Non parametric HMMs in the sense that the form of the emission distribution is not specified have only 
very recently been considered, since identifiability remained an open problem until  \citet{GaRo13} and  \citet{GaClRo13}, who prove a general identifiability result.
Because parametric modeling of emission distributions may lead to poor results in practice,
in particular for clustering purposes, recent interest in using non parametric HMMs appeared in applications, 
see  \citet{YaPaRoHo11}, \citet{GaClRo13} and references therein.
Theoretical results for estimation procedures in non parametric HMMs have also been obtained only very recently: \citet{DuCo12}
concerns regression models with hidden (markovian) regressors and unknown 
regression functions in Gaussian noise, and \citet{GaRo13} is about translated emission distributions.
%Identifiability for finite state space non parametric HMMs was finally solved by \citet{GaClRo13}. 
%It is proved in the latter that  non parametric HMMs may be fully
%identified provided the transition matrix of the hidden Markov chain has full rank, and the
%emission distributions are linearly independent. 

In this paper, we obtain posterior consistency results for Bayesian procedures in finite state space non parametric HMMs.
To our knowledge, this is the first result on posterior consistency in such models.
%about Bayesian procedures for finite state space non parametric HMMs.
In Section \ref{T}, we prove posterior consistency in terms of the weak topology and the $L_1$norm on
 marginal densities of consecutive observations. 
Our main result
%that the posterior asymptotically puts the mass on the neighborhood of the true density.
%In section \ref{T} we will work with neighborhoods with respect to the $L_1$ norm between two joint densities
%of $l$ consecutive observations.
is obtained under assumptions on the emission densities and on the prior %under which we prove posterior consistency 
which are very similar to the ones in the i.i.d. case, see Theorem \ref{th1}.
This result relies on a new control of the Kullback-Leibler divergence for HMMs, see Lemma \ref{KL}.
\textcolor{black}{Yet estimating the distribution of consecutive observations is not the main objective of a practitioner. Classifying the observations according to their corresponding hidden states or estimating the parameters of the model often are the questions of interest, see \citet{YaPaRoHo11}.}
In Section \ref{P} we build upon the recent identifiability result to deduce from Theorem \ref{th1}
posterior consistency for each component 
of the parameters. We obtain in general posterior consistency for the transition matrix of the Markov chain
and for the emission probability distribution in the weak topology, see Theorem \ref{th2}. \textcolor{black}{Stronger results are established in particular cases, see Corollary \ref{th3} and Theorem \ref{FD}.}
%Second, posterior consistency for the translation parameter and the translated density is shown for translated emission distributions, see Theorem \ref{th3}. 
Finally, some examples of priors that fulfill the assumptions \textcolor{black}{of Theorems \ref{th1} and \ref{th2}}
are studied in Section \ref{E}. 

Particularly in Section \ref{D} the discrete case is thoroughly studied with a Dirichlet process prior. %Unlike for density estimation in the i.i.d. case, direct calculations cannot be achieved. 
Sufficient   and almost necessary assumptions  to apply Theorem \ref{th1} are given in Proposition \ref{thD}. Moreover in this framework, posterior consistency of the marginal smoothing distributions, used in segmentation or classification, is derived in Theorem \ref{FD}.
% states posterior consistency for the so desired marginal smoothing which enables to classify the observations. 

All proofs are given in \textcolor{black}{Appendices} \textcolor{black}{\ref{kr} and \ref{op}}.

%%%%%%%%%%%%%%%%%%%%%%%%%%%%%%%%%%%%%%%%%%%%%%%%%%%%%%%%%%%%%%%%%%%%%%%%%%%%%%%%%%%%%%%%%%%%%%%%%%%%%%%%%
%%%%%%%%%%%%%%%%%%%%%%%%%%%%%%%%%%%%%%%%
%%%%%%%%%%%%%%%%%%%%%%%%%%%%%%%%%%%%%%%%%%%%%%%%%%%%%%%%%%%%%%%%%%%%%%%%%%%%%%%%%%%%%%%%%%%%%%%%%%%%%%%%
%%%%%%%%%%%%%%%%%%%%%%%%%%%%%%%%%%%%%%%
%%%%%%%%%%%%%%%%%%%%%%%%%%%%%%%%%%%%%%%%%%%%%%%%%%%%%%%%%%%%%%%%%%%%%%%%%%%%%%%%%%%%%%%%%%%%%%%%%%%%%%%%%%

\section{Settings and main Theorem}
\subsection{Notations}\label{N}

We now precise the model and give some notations.
Recall that finite state space HMMs are stochastic processes $(X_t,Y_t)_{t\in \mathbb{N}}$ 
such that $(X_t)_{t\in \mathbb{N}}$ is a Markov chain taking values in a finite set, and 
conditionally on $(X_t)_{t\in \mathbb{N}}$, the random variables $Y_t$, $t\in \mathbb{N}$, 
are independent. The distribution of $Y_{t}$ depending only on $X_{t}$ is called the emission distribution. 
The number $k$ of hidden states is known, so that the state space of the Markov chain is set
to $\{1, \dots, k\}$. Throughout the paper, for any integer $n$, an $n$-uple $(x_{1},\ldots,x_{n})$ is denoted $x_{1:n}$.

Let
$\Delta_k = \{(x_1,\dots,x_k)~ : ~ x_i\geq 0 , ~  i=1,\dots k  ~ ; ~ \sum_{i=1}^k x_i=1\}$ denote
the $k-1$-dimensional simplex. Let $Q$ denote the $k\times k$ transition matrix of the Markov chain, 
so that identifying $Q$ as the $k$-uple of transition distributions (the lines of the matrix), 
we write $Q\in \Delta_k^k$. We denote $\mu\in\Delta_k$  the initial probability measure, 
that is the distribution of $X_1$.
For $\underline{q}\geq 0$, we also define
$$
\Delta^k(\underline{q})=\{ Q \in \Delta_k^k ~ : ~ \min_{i,j\leq k } Q_{i,j} \geq \underline{q}\},
$$
so that $\Delta^k(0)=\Delta_k^k $.
We now recall some properties of Markov chains with transition matrix in $\Delta^k(\underline{q})$.
Note that $\underline{q}$ needs to be less than $\frac{1}{k}$ for $\Delta^k(\underline{q})$ to be non empty.
Then for all $Q$ in $\Delta^k(\underline{q})$, $\max_{i,j}Q_{i,j} \leq 1-(k-1)\underline{q}$.
%The set $\Delta^k(\underline{q})$ enables us to work with nice Hidden Markov chains which mix well.
Also, if $Q\in\Delta^k(\underline{q})$, then for any $i \in \{1, \dots,k\}$ and $A \subset \{1, \dots,k\}$,
 $\sum_{j\in A}Q_{i,j}\geq k \underline{q} u(A) $,
 with $u$ the uniform probability on $\{1, \dots,k\}$.
 Besides if $Q \in \Delta^k(\underline{q})$ with $\underline{q}>0$,
 the chain is irreducible, positive recurrent and admits a unique stationary probability measure
 denoted $\mu^Q$ for which $\underline{q} \leq \mu^Q(i) \leq 1- (k-1) \underline{q}$, $1\leq i \leq k$.

We assume that the observation space is $\mathbb{R}^{d}$ endowed with its Borel sigma field.
Let $\mathcal{F}$ be the set of probability density functions with respect to a reference measure
$\lambda$ on $\mathbb{R}^{d}$. $\mathcal{F}^k$ is the set of possible emission densities, 
that is for $f=(f_{1},\ldots,f_{k})\in\mathcal{F}^k$,
the distribution of $Y_{t}$ conditionally to $X_{t}=i$ will be $f_{i}\lambda$, $i=1,\ldots,k$. 
See Figure \ref{graph} for a visualization of the model.
\begin{figure}[ht]
	\centerline{\includegraphics[width=10cm]{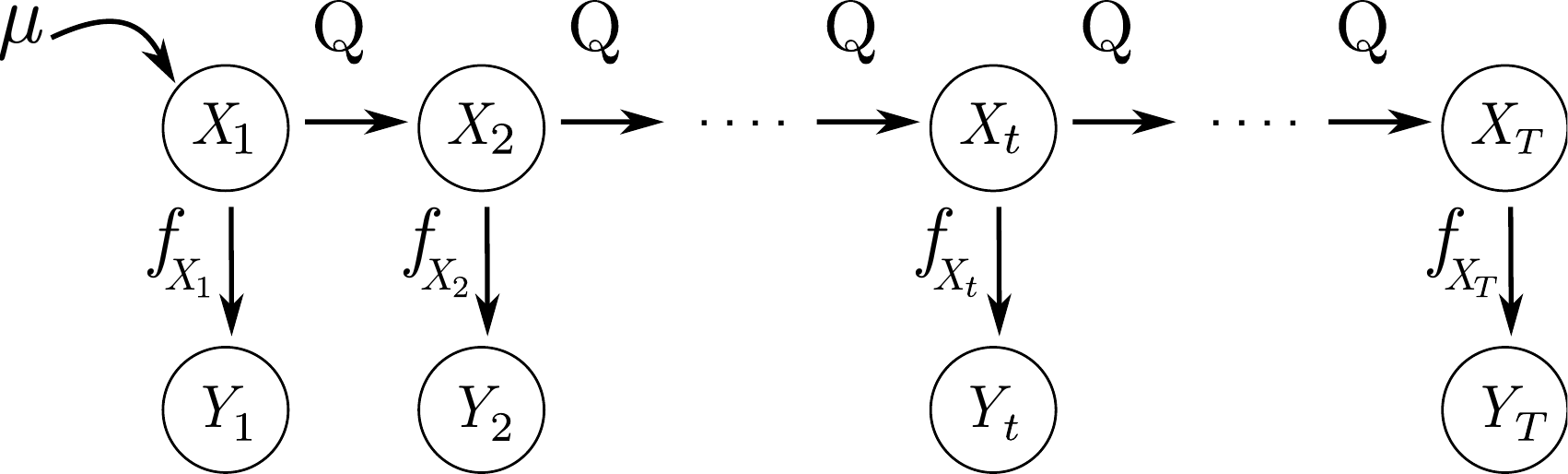}}
	\caption{The model}
	\label{graph}
\end{figure}

Let 
$$\Theta=\{ \theta = (Q,f) ~ :  ~ Q \in \Delta_k^k, f \in \mathcal{F}^k\}
$$ 
and
$$\Theta(\underline{q})=\{ \theta = (Q,f) ~ :  ~ Q \in \Delta^k(\underline{q}), f \in \mathcal{F}^k\}.
$$

Then $\mathbb{P}^\theta$ (resp. $\mathbb{P}^{\theta,\mu}$) denotes the probability distribution of
$(X_t,Y_t)_{t\in \mathbb{N}}$ under $\theta$ 
 and initial probability $\mu^\theta:=\mu^Q$ (respectively $\mu$). 
Let $p^\theta_l$ ($p^{\theta,\mu}_l$ resp.) denote the probability density of $Y_1, \dots, Y_l$ 
 with respect to $\lambda^{\otimes l}$ under $\mathbb{P}^\theta$ (resp. $\mathbb{P}^{\theta,\mu}$).
and $P_l^\theta$ ($P^{\theta,\mu}_l$ resp.) the marginal distribution of $Y_1, \dots, Y_l$ under $\mathbb{P}^\theta$ (resp. $\mathbb{P}^{\theta,\mu}$).
 So for any $\theta \in \Theta$, initial probability $\mu$, and 
 measurable set $A$ of $\{1,\dots,k\}^l\times (\mathbb{R}^d)^l  $:
 \begin{equation*}
\begin{split}
  \mathbb{P}^{\theta,\mu} &((X_{1:l},Y_{1:l}) \in A) \\
 &=\int \sum_{x_1, \dots, x_l=1}^k \mathds{1}_{(x_1, \dots,x_l,y_1,\dots,y_l)\in A} ~ \mu_{x_1} Q_{x_1,x_2} \dots
 Q_{x_{l-1},x_l} \\
  &\qquad f_{x_1}(y_1) \dots f_{x_l}(y_l) \lambda(dy_1)\dots\lambda(dy_l) ,
 \end{split}
 \end{equation*}
\begin{equation*}
\begin{split}
 & p^{\theta,\mu}_l  (y_1, \dots,y_l) = 
 \sum_{x_1, \dots, x_l=1}^k 
 \mu_{x_1} Q_{x_1,x_2} \dots Q_{x_{l-1},x_l} f_{x_1}(y_1) \dots f_{x_l}(y_l), 
 %\quad \lambda^{\otimes l} \text{ a.s.} 
 \end{split}
 \end{equation*}
 and
 $\displaystyle{P_l^{\theta, \mu}= p_l^{\theta, \mu} \lambda^{\otimes l}.}$

We denote by $\delta_\mu \otimes \pi$ the prior on $\Delta_k \times \Theta$,
where $\mu \in \Delta_k$ is an initial probability measure.
We assume that $\pi$ is a product of probability measures on $\Theta$,
$\pi= \pi_Q \otimes \pi_f$ such that
$\pi_Q$ is a probability distribution on $\Delta_k^k$ and
$\pi_f$ is a probability distribution on $\mathcal{F}^k$. 

We assume throughout the paper that the observations are distributed from $\mathbb{P}^{\theta^*}$
so that their distribution 
is a stationary HMM.
We are interested in posterior consistency, that is to prove that
with $\mathbb{P}^{\theta^*}$-probability one, for all neighborhood $U$ of $\theta^*$ : 
$$\lim_{n \to +\infty} \pi(U| Y_{1:n})= 1 .$$

The choice of a topology on the parameters arises here.
For any distance or pseudometric $D$, we denote $N(\delta,A,D)$ the $\delta$-covering number of the set $A$ with respect to $D$, 
that is the minimum number $N$ of elements $a_1, \dots, a_N$ such that
for all $a \in A$, there exists $n \leq N$ such that $D(a,a_n)\leq \delta$.

For $k \times k$ matrices $M$, we use 
$$\lVert M \rVert= \max_{1 \leq i,j \leq k} \lvert M_{i,j} \rvert.$$
For vectors $v$ in $\mathbb{R}^k$, we denote 
$$\lVert v \rVert_1 = \sum_{1\leq i \leq k} \lvert v_i \rvert.$$
For probabilities $P_1$ and $P_2$, let $p_1$ and $p_2$ be their respective 
densities with respect to some dominated measure $\nu$.
We use the total variation norm :
$$\lVert P_1-P_2 \rVert_{TV}= \frac{1}{2} \int \lvert p_1-p_2 \rvert d\nu = \frac{1}{2} \lVert p_1-p_2 \rVert_{L_1(\nu)}$$
and the Kullback-Leibler divergence :
\begin{equation*} 
 \begin{split}  
  KL(P_1,P_2) & =\left\{
    \begin{array}{ll}
      \int p_1 \log(\frac{p_1}{p_2}) d\nu & \text{ if } P_1 << P_2, \\
          + \infty                  & \text{ otherwise.}
    \end{array}
    \right.\\
     \end{split}
\end{equation*}
We  also denote  $KL(p_1,p_2)$ for $KL(p_1\nu,p_2\nu)$.
% We now want to choose a topology on the parameters.
% In the non-parametric i.i.d. case the $L_1$ norm and the weak topology on densities are used to prove 
% strong and weak consistency in the the framework of density estimation.
% Obviously comparing the density for one observation will not be enough, 
% since we we will not see the dependence on the observations.
% Then we generalize the strong topology used in the i.i.d. context
% by comparing joint densities of several consecutive observations.
On $\mathcal{F}^k$ we use the distance $d(\cdot,\cdot)$ defined for all $g=(g_1, \dots,g_k)$, 
$\tilde{g}=(\tilde{g_1}, \dots,\tilde{g_k}) $ by
$$d(g,\tilde{g}) = \max_{1 \leq j \leq k} \lVert g_j - \tilde{g_j} \rVert_{L_1(\lambda)} $$

On $\Theta(\underline{q})$, we use the following pseudometric for $l\geq3$, $l \in \mathbb{N}$,
$$D_l(\theta, \theta')=
\int | p_l^\theta (y_1, \dots, y_l) - p_l^{\theta'} (y_1, \dots, y_l)|\lambda(dy_1)\dots\lambda(dy_l)
= \lVert p_l^\theta - p_l^{\theta'} \rVert_{L_1(\lambda^{\otimes l})}.$$
Then a $D_l$-neighborhood of $\theta$ is a set which contains a set 
$\{ \theta' ~ : ~ D_l(\theta,\theta')<\epsilon\}$ for some $\epsilon>0$.
We also use the weak topology on marginal distributions $(P_l^\theta)_\theta$.
We recall that in any neighborhood of $P_l^\theta $ in the weak topology on probability measures
 there is a subset which is a union of sets of the form
 $$\left\{P ~ : ~ \left\lvert \int h_j dP - \int h_j p_l^\theta d\lambda^{\otimes l}  \right\rvert < \epsilon_j,  ~
  j=1,\dots,N \right\} ,$$
 where for all $1 \leq j \leq N$, $\epsilon_j>0$ and  $h_j$ is in the set $\mathcal{C}_b((\mathbb{R}^d)^l)$
 of all bounded continuous functions 
 from $(\mathbb{R}^d)^l$ to $\mathbb{R}$.
We prove posterior consistency in this general nonparametric context using this weak topology on marginal distributions $(P_l^\theta)_\theta$
and the $D_l$-pseudometric in Section \ref{T}.
We study the posterior consistency for the transition matrix and the emission probabilities separately
in Section \ref{P}.

%Recall that the support of a probability measure is the intersection of all closed sets of measure one.
%In a separable set, the support has probability one.
%Moreover if one point $x$ is in the support of a probability measure $\pi$ 
%then for all neighborhood $V$ of $x$, $\pi(V)>0$.
%% Indeed if there exists a neighborhood $V$ of $x$ such that $\pi(V)=0$
% there exists $W$ open set included in $V$ such that $x\in V$,
% then $\support(\pi)\backslash W$ is a closed set of measure $1$
% which is strictly included in $\support(\pi)$ : absurd !!

Finally the sign $\lesssim$ is used for inequalities up to a multiplicative constant possibly depending on fixed parameters.

%%%%%%%%%%%%%%%%%%%%%%%%%%%%%%%%%%%%%%%%%%%%%%%%%%%%%%%%%
%%%%%%%%%%%%%%%%%%%%%%%%%%%%%%%%%%%%%%%%%%%%%%%%%%%%%%%%
%%%%%%%%%%%%%%%%%%%%%%%%%%%%%%%%%%%%%%%%%%%%%%%%%%%%%%%

\subsection{Main Theorem}\label{T}

In this section we state our general theorem on posterior consistency 
for nonparametric hidden Markov models in the weak topology on marginal distributions $(P_l^\theta)_\theta$ and the $D_l$-topology.
We consider the following assumptions. Fix $l\geq 3$.
\vspace{5mm}

(A1) For all $\epsilon>0$ small enough there exists a set
 $\Theta_\epsilon \subset \Theta(\underline{q})$ such that $\pi(\Theta_\epsilon)>0$
  and for all $\theta=(Q,f) \in \Theta_\epsilon$,
 
 \quad (A1a)  $\displaystyle{\lVert Q - Q^* \rVert < \epsilon}$,
 
 \quad (A1b) $\displaystyle{\max_{1\leq i\leq k} \int f_i^*(y) \max_{1\leq j \leq k} \log \left(\frac{ f_j^*(y)}{f_j(y)} \right) \lambda(dy) < \epsilon}, $
%  there exist $\nu \in (0,1)$, $d \in \mathbb{N}$ and   $M_\epsilon \lesssim \epsilon^{-d}$ 
%  such that for all $i,j=1, \dots,k$,
%  $$\int f^*(y-m^*_i) \left(\frac{ f^*(y-m^*_j)}{f(y-m_j)} \right)^\nu dy \leq M_\epsilon, $$
 
 \quad (A1c) for all $y \in \mathbb{R}^d$ such that  $\displaystyle{\sum_{i=1}^k f^*_i(y)>0}$, $\displaystyle{  \sum_{j=1}^k f_j(y) >0}$,
 
 \quad (A1d) %$\displaystyle{\sup_{\theta \in \Theta_\epsilon} 
 $\displaystyle{\sup_{y  ~ : ~ \sum_{i=1}^k f^*_i(y)>0} \max_{1 \leq j \leq k} f_j(y) < +\infty}$
 
 \quad (A1e) $\displaystyle{\sum_{i=1}^k   \int f^*_i(y) \left\lvert 
 \log \left(  \sum_{j=1}^k f_j(y)\right) \right\rvert \lambda(dy) < + \infty}$
\vspace{5mm} 
 
 (A2) For all $n>0$, for all $\delta>0$
there exists a set $\mathcal{F}_n \subset \mathcal{F}^k$ and a real number $r_1>0$ such that 
$\pi_f\big(({\mathcal{F}_n})^c\big)\lesssim e^{-nr_1}$ 
and such that 
$$ \sum_{n>0}  N\left(\frac{\delta}{36l}, \mathcal{F}_n , d(\cdot,\cdot)\right)
\exp\left(-\frac{n \delta^2 k^2 \underline{q}^2}{32 l}\right)< + \infty.$$
\vspace{5mm}

\begin{theorem}\label{th1} Let $\underline{q}>0$.
Assume that the support of the prior $\pi$ is included in $\Theta(\underline{q})$
 and that for all $1\leq i\leq k$, $\mu_i \geq \underline{q}$.
 
 \begin{enumerate}[label=\alph*)]
 \item \label{a.} If Assumption (A1) holds then
 for all weak neighborhood $U$ of $P_l^{\theta^*}$, 
$$\mathbb{P}^{\theta^*}\left(\lim_{n \to \infty} \pi(U|Y_{1:n})=1 \right)=1. $$
 
\item \label{b.} Moreover if Assumptions (A1) and (A2) hold then, for all $\epsilon>0$, 
$$\mathbb{P}^{\theta^*}\left(\lim_{n \to \infty} \pi(~ \left\{ \theta: ~ D_l(\theta,\theta^*)<\epsilon  \right\} ~ |Y_{1:n})=1 \right)=1. $$
\end{enumerate}
\end{theorem}

\begin{remark}
We assume everywhere in the paper  that the support of the prior is included in $\Theta(\underline{q})$. It means the results of this paper can only be applied to priors $\pi_Q$ on transition matrices which vanish close to the border of $\Delta_k^k$. 
This assumption is satisfied by a product of truncated Dirichlet distribution i.e. if
the lines $Q_{i,\cdot}$ of $Q$ are independently distributed from a law proportional to:
$$ Q_{i,1}^{\alpha_1-1} \dots Q_{i,k}^{\alpha_k-1} \mathds{1}_{\left\{\underline{q} \leq Q_{i,j} \leq  1, ~ \forall 1\leq j \leq k\right\}}  dQ_{i,1} \dots dQ_{i,k}$$
where $\alpha_1, \dots, \alpha_k >0$.

\textcolor{black}{The restriction on $\Theta(\underline{q})$ comes from the test built in \citet{GaRo12}. On this set, HMMs are geometrically ergodic.  It is a common assumption in the literature see \citet{DoMa01}, \citet{Do04} or \citet{DoMoOlHa11} for instance. Besides \citet{GaRo12} explain the difficulty which appears when the Markov chain does not mix well. They are also able to obtain a less restrictive assumption on the support of the prior on transition matrices. In return they assume a more restrictive assumption on the log-likelihood, compare  Equations \eqref{kleq} and \eqref{klimeq} with their  Assumption C1 . }
\end{remark}

In the case of  density estimation with i.i.d. observations it is usual  to control the Kullback-Leibler support
of the prior to show weak posterior consistency and to control in addition a metric entropy to obtain strong consistency see Chapter 4 of \citet{GhRa03}.
Assumptions (A1) and (A2) are similar in spirit.
Assumption (A1) replaces the assumption on the true density function being in the Kullback-Leibler support
of the prior in the i.i.d. case. 
(A1a) ensures that the transition matrices of $\Theta_\epsilon$ are in a ball of radius $\epsilon$
around the true transition matrix.
Under (A1b) the emission densities are in an $\epsilon$ Kullback-Leibler ball around the true one.
(A1c), (A1d) and (A1e) are assumptions under which the log-likelihood converges 
$\mathbb{P}^{\theta^*}\text{\!\!\! -a.s.}$ and in $L_1(\mathbb{P}^{\theta^*})$.
(A2) is very similar to the assumptions of the metric entropy of Theorem 4.4.4 in \citet{GhRa03}.

%In the proof of \ref{th1} in Appendix \ref{Prth1}, notice that a stronger result is shown: 
%$$\mathbb{P}^{\theta^*}\left( \pi(U^c|Y_{1:n})\geq \exp(-nr) \text{ infinitely often }  \right)=0.$$ 
%is equivalent to 
%$$\mathbb{P}^{\theta^*}\left( \exists N, ~ \forall n\geq N, ~ \pi(U^c|Y_{1:n})\leq \exp(-nr)   \right)=1$$
%which implies that 
%$$\mathbb{P}^{\theta^*}\left( \lim_{n \to + \infty} \pi(U^c|Y_{1:n})= 0  \right)=1.$$ 
%Thus under Assumption (A1), the posterior is consistent for the weak topology
%on marginal distributions $(P_l^\theta)_\theta$. 
%If moreover (A2) holds then it is also consistent for the $D_l$-pseudometric.

%This similarity with the i.i.d. case is due to the same technique of proof.
In Appendix \ref{Prth1}, the proof  of Theorem \ref{th1} relies on the method of \citet{Ba88}. It consists in controlling Kullback-Leibler neighborhoods  and building tests. \textcolor{black}{The construction of tests is quite straightforward thanks to Rio's inequality \cite{Ri00} which generalizes Hoeffding's inequality. To prove \ref{a.}, we use the usual strategy presented in Section 4.4.1 in \citet{GhRa03} together with Rio's inequality \cite{Ri00} and \citet{GaRo13}. To prove \ref{b.} we use the tests of \citet{GaRo13}. To control the Kullback-Leibler neighborhoods, we use the following lemma whose proof is given  in Appendix \ref{AKL}}.

\begin{lemma}\label{KL} Let $\theta^*$ be in $\Theta(\underline{q})$.
 If (A1) holds  then for all $0<\epsilon<1$, there exists $N \in \mathbb{N}$ such that for all $n\geq N$ and for all $\theta \in \Theta_\epsilon$:
 $$\frac{1}{n} KL(\mathbb{P}^{\theta^*}_n,\mathbb{P}^{\theta,\mu}_n) \leq  
 \frac{3}{\underline{q}} ~ \epsilon .$$
\end{lemma}

\subsection{Consistency of each component of the parameter}\label{P}

In this Section we look at the consequences of Theorem \ref{th1} on posterior consistency for the transition matrix and
the emission probabilities separately. Estimating consistently the components of the parameter is of great importance. First  one may want to know the proportion of each population or the probability of moving from one population to another, i.e. the transition matrix. Secondly, these components are important to recover the smoothing distribution and then clustering the observations, see \citet{CaMoRy05} and Theorem \ref{FD}. 

The consistency of each component, i.e. the transition matrix and the emission distributions does not directly result from consistency of the marginal distribution of the observations, see \citet{DuCo12}. Obviously, identifiability seems to be necessary to obtain this implication yet it is not sufficient. We obtain posterior consistency for the components of the parameter thanks to the result of identifiability of \citet{GaClRo13}, an inequality linking the $D_l$ pseudometric to distances on each component of the parameter and an argument of compactness.

We use a product topology on the set of parameters.
 In particular we study consistency in the topology associated with the sup norm on transition matrices $\lVert \cdot \rVert$
 and the weak topology on probabilities for the emission probabilities up to label switching.
 To deal with label switching, we need the following definitions.
 Let $\mathcal{S}_k$ denote the symmetric group on $\{1, \dots, k\}$. Let $\sigma$ be a permutation  in $\mathcal{S}_k$, 
 for all matrices $Q \in \Delta^k_k$, we denote
 $\sigma Q$  the following matrix : for all $1 \leq i,j \leq k$,
 $$(\sigma Q)_{i,j}= Q_{\sigma(i),\sigma(j)}.$$If $(X_t,Y_t)_{t \in \mathbb{N}}$ is distributed from $P^{(Q,f)}$ and $\tilde{X}_t=\sigma^{-1}(X_t)$, for $\sigma \in \mathcal{S}_k$, then $(\tilde{X}_t,Y_t)_{t \in \mathbb{N}}$ is distributed from $P^{(\sigma Q, (f_{\sigma(1)}, \dots, f_{\sigma(k)}))}$, i.e the labels of the Markov chain have been switched.
Under the assumptions of Theorem \ref{th1} and of identifiability we prove that the posterior concentrates around $(Q^*,f^*)$ up to label switching, i.e. around $\{\sigma Q^* ,(f^*_{\sigma(1)}, \dots, f^*_{\sigma(k)})\}_{\sigma \in \mathcal{S}_k}$, in Theorem \ref{th2} whose proof is given in Appendix \ref{P:th2}. In other words we obtain posterior consistency considering neighborhoods of the form 
$$\left\{ \exists \sigma \in \mathcal{S}_k; ~ \sigma Q \in U_{Q^*}, ~ f_{\sigma(i)} \in U_{f^*_i},  i=1\dots k \right\} $$
where $U_{Q^*}$ is a neighborhood of $Q^*$ and for all $1 \leq i  \leq k$, $U_{f^*_i}$ is a weak neighborhood of $f^*_i \lambda$.
That is to say we consider the product of the sup norm topology on transition matrices and of the weak topology on the emission distributions up to label switching.
%In \citet{GaClRo13} it is proved that knowing the distribution of three consecutive observations $Y_1,Y_2,Y_3$
%enables to recover $Q,f_1,\dots,f_{k-1}$ and $f_k$ up to label swapping if 
%$Q$ has full rank and $f_1 \lambda,\dots,f_k \lambda$ are linearly independent.
%Then under the latter assumptions, it is sufficient to work with the weak topology on marginal distributions $(P_l^\theta)_\theta$ and the $D_l$ pseudometric for $l\geq3$ 
%to distinguish two parameters $\theta$ and $\theta'$ up to label swapping.}

\begin{theorem}\label{th2}
Let $\theta^*=(Q^*,f^*)$. Suppose $f^*_1 \lambda, \dots,f^*_k \lambda$ are linearly independent 
and $Q^*$ has full rank.
 Let $\underline{q}>0$,
 assume that  $\mu_i\geq\underline{q}$, that the support of the prior $\pi$ is included in $\Theta(\underline{q})$ and that (A1) and (A2) hold.

 Then for all weak neighborhood $U_{f^*_i}$ of $f^*_i \lambda$, for all $1\leq i \leq k$
and for all neighborhood $U_{Q^*}$ of $Q^*$,

\begin{equation}\label{sp}
\mathbb{P}^{\theta^*}\left(\lim_{n \to +\infty} \pi\bigg(
\left\{ \exists \sigma \in \mathcal{S}_k; ~ \sigma Q \in U_{Q^*}, ~ f_{\sigma(i)} \in U_{f^*_i},  ~  i=1\dots k \right\}   \bigg|  ~ Y_{1:n}\bigg)
= 1 \right)=1. 
\end{equation}
\end{theorem}

\begin{remark}
In particular, Equation \eqref{sp} implies that for all $\epsilon >0$
$$\mathbb{P}^{\theta^*}\left(\lim_{n \to +\infty} \pi\left( ~ 
\bigcup_{\sigma \in \mathcal{S}_k}  \left\{Q :  \lVert Q - \sigma Q^* \rVert < \epsilon \right\} ~  \bigg| ~ Y_{1:n}\right)
= 1 \right)=1 .$$
It means that under the assumptions of Theorem \ref{th2}, the posterior concentrates around $\{\sigma Q^*, \sigma \in \mathcal{S}_k\}$.
Equation \eqref{sp} also implies that for all $N \in \mathbb{N}$, for all $h_i \in \mathcal{C}_b(\mathbb{R}^d)$, for all $\epsilon_i>0,~ 1 \leq i\leq N,$
$$\mathbb{P}^{\theta^*}\Bigg(\lim_{n \to +\infty} \pi\bigg( ~ \bigcup_{\sigma \in \mathcal{S}_k} \left\{P:  \left\lvert \int h_i dP- \int h_i f_{\sigma(j)}^*d\lambda \right\rvert < \epsilon_i \right\} \bigg| ~ Y_{1:n}\bigg) =1 \Bigg)=1.$$
This last result is a weak result which allows to consistently recover smooth functionals of the emission distributions $(f^*_j)_j$. We obtain stronger results in Sections \ref{TEP} and \ref{D}. 
\end{remark}

% \subsection{Discrete HMMs}
% 
% In the case where the support of $\lambda$ is included in $\mathbb{N}$ that is to say we consider 
% discrete emission probabilities, we deduce consistency for each components of the parameters 
% with the strong topology on the emission probabilities. Indeed with discrete probabilities, 
% weak and strong convergences are the same.
% 
% \begin{corollary}
% Let $\lambda$ be such that its support is included in $\mathbb{N}$.
%  Let $\theta^*=(Q^*,f^*)$ such that $f^*_1 d\lambda, \dots,f^*_k d\lambda$ are linearly independent 
% and $Q^*$ has full rank.
%  Let $\underline{q}>0$,
%  such that  $\mu_i>\underline{q}$, the support of the prior $\pi$ is included in $\Theta(\underline{q})$,
%   (A2) and (A1).
%   
% Then for all strong neighborhood $U_{f^*_i}$ of $f^*_i$, for all $1\leq i \leq k$
% and neighborhood $U_{Q^*}$ of $Q^*$,
% 
% $$\mathbb{P}^{\theta^*}\left(\lim_{n \to +\infty} \pi\left(
% \cup_{\sigma \in \mathcal{S}_k} \sigma(U_{Q^*})\times U_{f^*_{\sigma(1)}}\times \dots \times U_{f^*_{\sigma(k)}}    |Y_{1:n}\right)
% = 1 \right)=1. $$
% \end{corollary}

%%%%%%%%%%%%%%%%%%%%%%%%%%%%%%%%%%%%%%%%%%%%%%%%%%%%%%%%%%%%%%%%%%%%%%%%%%%%%%%%%%%%%%%%%%%%%%%%%%%%%%%%%%
%%%%%%%%%%%%%%%%%%%%%%%%%%%%%%%%%%%%%%%%
%%%%%%%%%%%%%%%%%%%%%%%%%%%%%%%%%%%%%%%%%%%%%%%%%%%%%%%%%%%%%%%%%%%%%%%%%%%%%%%%%%%%%%%%%%%%%%%%%%%%%%%%%%
%%%%%%%%%%%%%%%%%%%%%%%%%%%%%%%%%%%%%%%%%%%%
%%%%%%%%%%%%%%%%%%%%%%%%%%%%%%%%%%%%%%%%%%%%%%%%%%%%%%%%%%%%%%%%%%%%%%%%%%%%%%%%%%%%%%%%%%%%%%%%%%%%%%%%%%%

\section{Examples of priors on $f$}\label{E}

In this section we apply Theorems \ref{th1} and \ref{th2}  for different types of priors and emission models.
In Section \ref{MG} we deal with emission probabilities which are independent mixtures of Gaussians.
Translated emission probabilities are studied in Section \ref{TEP}. 
\textcolor{black}{Finally we consider the discrete case with Dirichlet process priors in Section \ref{D}.}

Assumptions (A1) and (A2) are purposely designed to resemble the types of assumptions found in density estimation for i.i.d. observations. This allows us to use existing results on consistency in the case of i.i.d. observations.
%The similarity of Assumptions (A1) and (A2) with the methods used to obtain weak and strong consistency in the case of density estimation enable us to use applications of these methods in the framework of hidden Markov models. 
This is done in Sections \ref{MG} and \ref{TEP} following \citet{To06}. Contrariwise we develop a new method to deal with \textcolor{black}{the Dirichlet process prior} for the discrete case in Section \ref{D}.

%%%%%%%%%%%%%%%%%%%%%%%%%%%%%%%%%%%%%%%%%%%%%%%%%%%
%%%%%%%%%%%%%%%%%%%%%%%%%%%%%%%%%%%%%%%%%%%%%%%%%%%
%%%%%%%%%%%%%%%%%%%%%%%%%%%%%%%%%%%%%%%%%%%%%%%%%%%%

\subsection{Independent mixtures of Gaussians}\label{MG}

We consider the well known location-scale mixture of Gaussian distributions as prior model for each $f_i$,
namely each density under the prior is written as 
\begin{equation}\label{eqMG:normal}
g(y)=\int_{\mathbb{R}\times (0,+ \infty)} \phi_\sigma(y-z)dP(z,\sigma)=: \phi *P
\end{equation}
where $\phi_\sigma$ is the Gaussian density with mean zero and variance $\sigma^2$ and 
$P$ is a probability measure on $\mathbb{R}\times (0,+\infty)$. In this part, $\lambda$ is the Lebesgue measure on $\mathbb{R}$.
Let $\pi_P$ be a probability measure on the set of probability measures on $\mathbb{R}\times (0,+\infty)$.
Denote $\pi_g$ the distribution of $g$ expressed as \eqref{eqMG:normal} when $P \sim \pi_P$.
Then we consider the prior distribution on $f=(f_1, \dots, f_k)$ defined by $\pi_f=\pi_g^{\otimes k}$.
We need the following assumptions to apply Theorem \ref{th1} and \ref{th2}:

\quad (B1) 
 $$\pi_P\left(P~ :~ \int \frac{1}{\sigma} dP(z,\sigma) < \infty \right) =1 ,$$

\quad (B2) for all $1\leq j \leq k$, 
$f^*_j$ is positive, continuous on $\mathbb{R}$ and bounded by $M<\infty$,

\quad (B3) for all  $1\leq i \leq k$,
$$\left\lvert \int_{\mathbb{R}} f^*_i(y) \max_{ 1 \leq j \leq k} \log(f^*_j(y)) \lambda(dy) \right\rvert <\infty$$

\quad (B4) for all  $1\leq i \leq k$, $1\leq j \leq k$,
$$ \int_{\mathbb{R}} f^*_i(y)  \log\left(\frac{f^*_j(y)}{\psi_j(y)}\right) \lambda(dy) <\infty$$
where $\psi_j(y)= \inf_{t\in [y-1,y+1]} f^*_j(t)$.

\quad (B5) for all  $1\leq i \leq k$, there exists $\eta>0$ such that 
$$\int_{\mathbb{R}} \lvert y \rvert^{2(1+\eta)} f^*_i(y) \lambda(dy) <\infty.$$

\quad (B6) for all $\beta>0$, $\kappa>0$, there exist a real number $\beta_0>0$, two increasing 
and positive sequences  $a_n$ and $u_n$ tending to $+\infty$ and a sequence 
 $l_n$ decreasing to $0$ such that 
 $$\pi_P\bigg(P~ :~P((-a_n,a_n] \times (l_n,u_n]) < 1- \kappa \bigg) \leq \exp(-n\beta_0), $$
  $$\text{with }\qquad \frac{a_n}{l_n} \leq n\beta, \qquad \log\left(\frac{u_n}{l_n}\right)  \leq n \beta $$

\begin{proposition}\label{th:MG}
Let $\underline{q}>0$.
Assume  that the support of the prior $\pi$ is included in $\Theta(\underline{q})$
 and that for all $1\leq i\leq k$, $\mu_i\geq\underline{q}$. Assume that $Q^*$ is in the support of $\pi_Q$ 
and that the weak support of $\pi_P$ contains all probability measures that are compactly supported.

Then
\begin{itemize}
\item  (B1), (B2), (B3), (B4), (B5) imply (A1)
\item and (B6) implies (A2).
\end{itemize}

%for all weak neighborhood $U$ of $P_l^{\theta^*}$, there exists $r>0$ such that
%$$\mathbb{P}^{\theta^*}\left( \pi(U^c|Y_{1:n})\geq \exp(-nr) \text{ infinitely often }  \right)=0. $$
%\vspace{5mm}
%If moreover we assume
%\newline
%then, for all $D_l$-neighborhood $U$ of $\theta^*$, there exists $r>0$ such that
%$$\mathbb{P}^{\theta^*}\left( \pi(U^c|Y_{1:n})\geq \exp(-nr) \text{ infinitely often }  \right)=0. $$
%
%\vspace{5mm}
%If all the previous assumptions hold,
%$f^*_1 \lambda, \dots,f^*_k \lambda$ are linearly independent 
%and $Q^*$ has full rank,
% then for all weak neighborhood $U_{f^*_i}$ of $f^*_i \lambda$, for all $1\leq i \leq k$
%and for all neighborhood $U_{Q^*}$ of $Q^*$,
%
%$$\mathbb{P}^{\theta^*}\left(\lim_{n \to +\infty} \pi\left(
%\cup_{\sigma \in \mathcal{S}_k} \sigma(U_{ Q^*})\times U_{f^*_{\sigma(1)}}\times \dots \times U_{f^*_{\sigma(k)}}    |Y_{1:n}\right)
%= 1 \right)=1. $$
\end{proposition}

%\begin{remark}
%Posterior consistency for the weak convergence on marginal distributions $P_l^\theta$, for the $D_l$-pseudometric, for each components of the parameters $Q$, $m$ and weak convergence on $f$ up to label switching are easy corollaries of Theorem \ref{th1} and Corollary \ref{th2} thanks to Proposition \ref{th:MG}.
%\end{remark}

 In particular in the case of the Dirichlet process mixture $DP(\alpha G_0)$ with base measure $\alpha G_0$, where
 $G_0$ is a probability measure on $\mathbb{R}\times (0,+\infty)$ and $\alpha>0$,
 Assumption (B1) holds as soon as 
 \begin{equation}\label{eqrk}
  \int_{\mathbb{R}\times (0,+\infty)} \frac{1}{\sigma} G_0(dz,d\sigma)<+\infty.
 \end{equation}
Indeed,
\begin{equation*}\begin{split}
\int \int \frac{1}{\sigma} P(dz,d\sigma) \pi_P(dP) 
& = \int \int \int_{\mathbb{[\sigma,+\infty)}} \frac{1}{t^2} \lambda(dt) P(dz,d\sigma) \pi_P(dP) \\
& = \int \frac{1}{\sigma}  G_0(dz,d\sigma).
\end{split}
\end{equation*}

Moreover Assumption (B6) easily holds as soon as 
 for all $\beta>0$, there exist a real number $\beta_0>0$,two increasing and positive sequences 
 $a_n$ and $u_n$ tending to $+\infty$ and a sequence
 $l_n$ decreasing to $0$ such that 
 \begin{equation}\label{eqrk2}
 \begin{split}
  G_0\left(\left(-a_n,a_n]\times (l_n,u_n]\right)^c \right)\leq \exp(-n \beta_0) \\
  \frac{a_n}{l_n}  \leq n\beta, \qquad \log\left(\frac{u_n}{l_n}\right) \leq n \beta
 \end{split}
 \end{equation}
 are verified (see Remark 3.1 of \citet{To06}).

% \begin{proof}[Proof of Remark \ref{rk}]
% Indeed 
% %
% \begin{equation*}\begin{split}
%\int \int \frac{1}{\sigma} P(dz,d\sigma) \pi_P(dP) 
%& = \int \int \int_{\mathbb{[\sigma,+\infty)}} \frac{1}{t^2} dt P(dz,d\sigma) \pi_P(dP) \\
%& =  \int \int_{\mathbb{R}} \frac{1}{t^2}  P(\sigma \in (0,t]) dt  \pi_P(dP) \\
%& =   \int_{\mathbb{R}} \frac{1}{t^2}  G_0(\mathbb{R} \times (0,t]) dt \\
%& = \int \frac{1}{\sigma}  G_0(dz,d\sigma)
%\end{split}
%\end{equation*}
%%
%which implies that (B1) holds using \eqref{eqrk}.
%
%Moreover with $G_n = G_0\left((-a_n,a_n]\times (l_n,u_n]\right)$,
%\begin{equation*}
%\begin{split}
%\pi_P &  \bigg(P~ :~P((-a_n,a_n] \times (l_n,u_n]) < 1- \kappa \bigg)\\
%   & = \frac{\Gamma(\alpha)}{\Gamma\left(\alpha G_n\right)
%                              \Gamma \left(  \alpha   (1-  G_n)\right)}
%                                  \int_0^{1-\kappa} x ^{\alpha G_n -1} (1-x)^{\alpha-\alpha G_n -1} dx\\
%  & \lesssim  \frac{ \kappa^{-1} \Gamma(\alpha)}{\Gamma\left(\alpha G_n\right)
%                              \Gamma \left(  \alpha   (1-  G_n)\right)} \int_0^{1-\kappa} x ^{\alpha G_n -1} dx \\
%  & \lesssim  \frac{1-G_n}{G_n} (1-\kappa)^{\alpha G_n} \\
%  & \lesssim G_0\left(\left(-a_n,a_n]\times (l_n,u_n]\right)^c \right)
%\end{split}
%\end{equation*}
% which implies that (B6) holds using \eqref{eqrk2}.
%\end{proof}

%%%%%%%%%%%%%%%%%%%%%%%%%%%%%%%%%%%%%%%%%%%%%%%%%%%%
%%%%%%%%%%%%%%%%%%%%%%%%%%%%%%%%%%%%%%%%%%%%%%%%%%%
%%%%%%%%%%%%%%%%%%%%%%%%%%%%%%%%%%%%%%%%%%%%%%%%%%%

\subsection{Translated emission probabilities}\label{TEP}

In this section we consider the special case of translated emission distributions that is to say for all $1 \leq j \leq k$, 
$$f_j(\cdot)=g( \cdot - m_j)$$ 
where $g$ is a density function on $\mathbb{R}$ with respect to $\lambda$
and for all $1 \leq j \leq k$, $m_j$ is in $\mathbb{R}$. In this part, $\lambda$ is still the Lebesgue measure on $\mathbb{R}$ and $d=1$\textcolor{black}{.}
This model has been in particular considered by \citet{YaPaRoHo11} for the analysis of genomic copy number variation.
First a corollary of Theorem \ref{th2} is given.
Then the particular case of location-scale mixture of Gaussians on $g$ is studied.

Let 
$$\Gamma=\{ \gamma=(Q,m,g), Q \in\Delta^k_k, m \in \mathbb{R}^k,m_1=0<m_2< \dots< m_k, g\in \mathcal{F} \}$$
 and 
 $$\Gamma(\underline{q})=\{ \gamma=(Q,m,g) \in \Gamma, Q \in \Delta^k(\underline{q}) \}.$$

 To $\gamma=(Q,m,g) \in \Gamma$ we associate $\theta=(Q,(g(\cdot-m_1),\dots,g(\cdot-m_k))) \in \Theta$.
We then denote $\mathbb{P}^\gamma$ for $\mathbb{P}^\theta$.
We assume that $\pi_f$ is a product of probability measure,
$$\pi_f=\pi_m \otimes \pi_g$$
where $\pi_g$ is a distribution on $\mathcal{F}$
and $\pi_m$ is a probability measure on $\mathbb{R}^k$. 
Note that under $\Gamma$, the model is completely identifiable, see Theorem 2.1 of \citet{GaRo13}.
 The uncertainty we had until now because of the label switching is resolved here. In Corollary \ref{th3} additionally to posterior consistency for the transition matrices, we obtain posterior consistency for the parameters of \textcolor{black}{translation} $m_j$ and for the weak convergence on the translated probability $g\lambda$. 
Under a stronger assumption, we get posterior consistency for the $L_1$-topology on the translated probability.
  
   Fix $l\geq 3$.
The following assumption replaces (A2) in the context of translated emission probabilities:

% \vspace{5mm}
%(C1) for all $\epsilon>0$ small enough there exists a set
% $\Gamma_\epsilon \subset \Gamma(\underline{q})$ such that $\pi(\Theta_\epsilon)>0$
%  and for all $\gamma=(Q,m,g) \in \Gamma_\epsilon$,
% 
% \quad (C1a)  $\displaystyle{\lVert Q - Q^* \rVert < \epsilon}$,
% 
% \quad (C1b) $\displaystyle{\max_{1\leq i\leq k} \int g^*(y-m^*_i) \max_{1\leq j \leq k} \log \left(\frac{ g^*(y-m^*_j)}{g(y-m_j)} \right) dy < \epsilon}, $
%%  there exist $\nu \in (0,1)$, $d \in \mathbb{N}$ and   $M_\epsilon \lesssim \epsilon^{-d}$ 
%%  such that for all $i,j=1, \dots,k$,
%%  $$\int f^*(y-m^*_i) \left(\frac{ f^*(y-m^*_j)}{f(y-m_j)} \right)^\nu dy \leq M_\epsilon, $$
% 
% \quad (C1c) for all $y \in \mathbb{R}^d$ such that there exists $1\leq j \leq k$ such that $g^*(y-m^*_j)>0$, $\displaystyle{ \frac{1}{k} \sum_{i=1}^k g(y-m_i) >0}$,
% 
% \quad (C1d) %$\displaystyle{\sup_{\theta \in \Theta_\epsilon} 
% $\displaystyle{\sup_{y  ~ : ~ \exists j,  ~ g^*(y-m^*_j)>0} g(y) < +\infty}$
% 
% \quad (C1e) $\displaystyle{\sum_{i=1}^k \mu^*_i \int g^*(y-m^*_i) \left\lvert 
% \log \left( \frac{1}{k} \sum_{j=1}^k g(y-m_j)\right) \right\rvert dy < + \infty}$ and
%  
% 
% \vspace{5mm}
(C2) for all $n>0$, for all $\delta>0$
there exists a set $\mathcal{F}_n \subset \mathbb{R}^k \times \mathcal{F}$ and a real number $r_1>0$ such that 
$\pi_f\big(({\mathcal{F}_n})^c\big)\lesssim e^{-nr_1}$ 
% Ici mieux expliciter cette distance !!!
$$ \sum_{n>0}  N\left(\frac{\delta}{36l}, \mathcal{F}_n , d(\cdot,\cdot)\right)
\exp\left(-\frac{n \delta^2 k^2 \underline{q}^2}{32 l}\right)< + \infty.$$

\begin{corollary}\label{th3}
 Let $\gamma^*=(Q^*,m^*,g^*)$ be in $\Gamma(\underline{q})$. Suppose $m^*_1=0<m^*_2< \dots< m^*_k $ 
and $Q^*$ has full rank.
 Let $\underline{q}>0$,
 assume that  $\mu_i\geq\underline{q}$, that the support of the prior $\pi$ is included in $\Gamma(\underline{q})$, that (A1) is verified with $f_j(\cdot)=g(\cdot-m_j), ~ 1\leq j \leq k$ and (C2) holds.

%Then for all neighborhood $U_{Q^*}$ of $Q^*$, for all neighborhood $U_{0}$ of $0\in \mathbb{R}^d$ and
% for all weak neighborhood $U_{g^*}$ of $g^*\lambda$,
%\begin{equation*}
%\begin{split}
% \mathbb{P}^{\gamma^*}\Bigg(\lim_{n \to +\infty} \pi\bigg(
%\cup_{\sigma \in \mathcal{S}_k} \sigma(U_{Q^*})\times U_0 \times 
%(U_0 +m^*_{\sigma(2)}-m^*_{\sigma(1)})\times \dots \\
%\times (U_0+m^*_{\sigma(k)}-m^*_{\sigma(1)})
%\times (\tau_{m^*_{\sigma(1)}}(U_{g^*}))   |Y_{1:n}\bigg)
%= 1 \Bigg)=1.
%\end{split}
%\end{equation*}
Then for all $\epsilon>0$, 
$$\mathbb{P}^{\gamma^*}\Big(\lim_{n \to +\infty} \pi(\left\{Q: \lVert Q- Q^* \rVert < \epsilon \right\} \big|  ~ Y_{1:n}) =1 \Big)=1,$$
$$\mathbb{P}^{\gamma^*}\Big(\lim_{n \to +\infty} \pi(\left\{m: \forall 1 \leq j \leq k, ~ \lvert m_j- m^*_j \rvert < \epsilon \right\} \big| ~  Y_{1:n}) =1 \Big)=1,$$
and for all $N \in \mathbb{N}$, for all $h_i \in \mathcal{C}_b(\mathbb{R}^d)$, for all $\epsilon_i>0,~ 1 \leq i\leq N,$
$$\mathbb{P}^{\gamma^*}\Bigg(\lim_{n \to +\infty} \pi\bigg(\left\{P:  \left\lvert \int h_i dP- \int h_i g^*d\lambda \right\rvert < \epsilon_i \right\} \bigg| ~  Y_{1:n}\bigg) =1 \Bigg)=1.$$

If moreover $\max_{1 \leq j \leq k} \mu^*_j>1/2$ and $g^*$ is uniformly continuous, then for all $\epsilon>0$,
$$\mathbb{P}^{\gamma^*}\Big(\lim_{n \to +\infty} \pi\left(\left\{g: \lVert g- g^* \rVert_{L_1(\lambda)} < \epsilon \right\} | Y_{1:n}\right) =1 \Big)=1.$$
\end{corollary}

The proof of Corollary \ref{th3}, in Appendix \ref{P:th3}, relies on the identifiability result of \citet{GaRo13} and the technique of proof of Theorem \ref{th2}.

\vspace{0.5cm}
%%%
%%%
%%%
%%%
In the same way as in Section \ref{MG}, we propose to apply Theorem \ref{th1} and Corollary \ref{th3} to a prior based on location-scale mixtures of Gaussians.
In this part we study a particular prior on the translated emission \textcolor{black}{density} $g$ which is the location-scale mixture of Gaussians.
\textcolor{black}{Then} $g$ is a sample drawn from $\pi_{g}$ if 
$$g(y)=\int_{\mathbb{R}\times (0,+\infty)} \phi_\sigma(y-z)dP(z,\sigma)$$
where $P$ is a sample drawn from $\pi_P$ and $\pi_P$ is a probability measure on probability measures on $\mathbb{R}\times(0, + \infty)$.
The following assumption help in proving (C2):

\quad (D6) for all $\beta>0$, $\kappa>0$, there exist a real number $\beta_0>0$, three increasing sequences of positive numbers
$m_n$, $a_n$ and $u_n$ tending to $+\infty$ and a sequence
 $l_n$ decreasing to $0$ such that 
 $$\pi_P\bigg(P~ :~P((-a_n,a_n] \times (l_n,u_n]) < 1- \kappa \bigg) \leq \exp(-n\beta_0), $$
$$\pi_m\bigg( ([-m_n,m_n]^k)^c\bigg) \leq exp(-n \beta_0),$$
  $$\frac{a_n}{l_n} \leq n\beta, \qquad \log\left(\frac{u_n}{l_n}\right) \leq n \beta, \qquad   \log\left(\frac{m_n}{l_n}\right) \leq n\beta$$

\begin{proposition}\label{th:tGM}
Let $\underline{q}>0$ and $\gamma^*$ in $\Gamma(\underline{q})$.
Assume  that the support of the prior $\pi$ is included in $\Gamma(\underline{q})$
 and that for all $1\leq i\leq k$, $\mu_i \geq \underline{q}$. Assume that $Q^*$ is in the support of $\pi_Q$, that $m^*$ is in the support of $\pi_m$ and that  the weak support of $\pi_P$ contains all probability measures that are compactly supported.

If (B1) is verified and (B2), (B3), (B4) and (B5) are verified with $f_j(\cdot)=\textcolor{black}{g}(\cdot-m_j), ~ 1 \leq j \leq k$ then (A1) holds.

Moreover (D6) implies (C2).
\end{proposition}

The proof of Proposition \ref{th:tGM} is very similar to that of Proposition \ref{th:MG} 
and is given in Appendix \ref{PtGM}.

%\begin{remark}
%Posterior consistency for the weak convergence on marginal distributions $P_l^\theta$, for the $D_l$-pseudometric, for each components of the parameters $Q$, $m$ and weak and $L_1$ convergence on $f$ are easy corollaries of Theorem \ref{th1} and Corollary \ref{th3} thanks to Proposition \ref{th:tGM}.
%\end{remark}

%%%%%%%%%%%%%%%%%%%%%%%%%%%%%%%%%%%%%%%%%%%%%%%%%%%%%
%%%%%%%%%%%%%%%%%%%%%%%%%%%%%%%%%%%%%%%%%%%%%%%%%%%%
%%%%%%%%%%%%%%%%%%%%%%%%%%%%%%%%%%%%%%%%%%%%%%%%%%%%

\subsection{Independent discrete emission distributions}\label{D}

Discrete emission probabilities, i.e. when the support of $\lambda$ is included in $\mathbb{N}$, have been
successfully used, for instance in genomics in \citet{GaClRo13}.

Note that for discrete emission probabilities, weak and $l_1$ convergences are the same so that weak posterior convergence implies 
$l_1$ posterior consistency. Thus Assumption (A2) becomes unnecessary in Theorems \ref{th1} and \ref{th2}. \textcolor{black}{ Moreover posterior consistency for the emission distributions in the weak topology  in Theorem \ref{th2} implies posterior consistency for the emission distributions in $l_1$.}

\textcolor{black}{In the discrete case, we prove in Appendix \ref{P:FD} that posterior consistency for the marginal probability of finitely many observations , for the transition matrix and for the emission distributions in $l_1$ together with the restriction of the prior on $\Delta^k(\underline{q})$ imply posterior consistency for the marginal smoothing:}

\begin{theorem}\label{FD}
\textcolor{black}{
Let $\underline{q}>0$.
Assume that the support of the prior $\pi$ is included in $\Theta(\underline{q})$
 and that for all $1\leq i\leq k$, $\mu_i\geq\underline{q}$.
If  $f^*_1 \lambda, \dots,f^*_k \lambda$ are linearly independent, $Q^*$ has full rank and (A1) holds
then for all finite integer $m$,}
\begin{equation*}
\begin{split}
\lim_{n \to +\infty} \pi\bigg(
\max_{1 \leq a_{1:m}\leq k} \lvert  P^\theta(X_{1:m}=a_{1:m}~|Y_{1:n}) \\
&\hspace{-4cm} - P^{\theta^*}(X_{1:m}=a_{1:m} ~|~Y_{1:n})\rvert <\epsilon   |Y_{1:n}\bigg)
= 1  \text{ in } P^{\theta^*}\text{-probability}.
\end{split}
\end{equation*}
\end{theorem}

In the \textcolor{black}{following} we apply Theorems \ref{th1}, \ref{th2} and \ref{FD} to a specific prior on the set of probability measures on $\mathbb{N}$ in the case of a HMM with discrete emission distributions. We consider a Dirichlet process $DP(\alpha G_0)$ 
with $\alpha$ a positive number and $G_0$ some probability measure on $\mathbb{N}$.
We then consider a prior probability measure on $\Theta$ defined by
$$\pi=\pi_Q \otimes DP(\alpha G_0)^{\otimes k}.$$

In Proposition \ref{thD}, we give sufficient and amost necessary conditions to obtain (A1).
Proposition \ref{thD} is proved in Appendix \ref{P:thD}.

\begin{proposition}\label{thD}
Let $\underline{q}>0$.
Assume that the support of the prior $\pi$ is included in $\Theta(\underline{q})$, that $Q^*$ is in the support of $\pi_Q$
 and that for all $1\leq i\leq k$, $\mu_i\geq\underline{q}$.

If
$$ \text{ (E1) for all }1 \leq i \leq k, ~
\sum_{l \in \mathbb{N}} \frac{f^*_i(l)}{G_0(l)} < + \infty$$
 then (A1) holds. 

Moreover if 
$$\text{(T) for all } 1\leq i \leq k, \sum_{l \in \mathbb{N}} f^*_i(l) (-\log f^*_i(l)) <+ \infty .$$
\newline
then (A1b) implies (E1).
%for all $\epsilon>0$
%$$(E1b) ~ DP(\alpha G_0)^{\otimes k }\left(\left\{f \in \mathcal{F}^k , \forall i  \sum_{l} f^*_i(l)\max_{1\leq j \leq k} \log \frac{f_j^*(l)}{f_j(l)}<\epsilon\right\}\right)>0$$
%\newline
%if and only if 
%\begin{equation*}
%(E1)\text{ for all } 1\leq i \leq k,\qquad \sum_{l \in \mathbb{N}} \frac{f^*_i(l)}{G_0(l)} < \infty.
%\end{equation*}
%and for all $D_l$-neighborhood $U$ of $\theta^*$, there exists $r>0$ such that
%$$\mathbb{P}^{\theta^*}\left( \pi(U^c|Y_{1:n})\geq \exp(-nr) \text{ infinitely often }  \right)=0. $$

% \vspace{5mm}
%If moreover  $f^*_1 \lambda, \dots,f^*_k \lambda$ are linearly independent 
%and $Q^*$ has full rank.
%
% Then  for all $\epsilon>0$,
%
%\begin{equation*}
%\begin{split}
%\mathbb{P}^{\theta^*}\Bigg(\lim_{n \to +\infty} \pi\bigg(
%\cup_{\sigma \in \mathcal{S}_k} \{\lVert Q-Q^* \rVert < \epsilon \times \{\lVert f_1-f^*_{\sigma(1)}\rVert_{l_1}\}\times \dots \\
% &\hspace{-5cm}\times \{\lVert f_k-f^*_{\sigma(k)}\rVert_{l_1}\}    |Y_{1:n}\bigg)
%= 1 \Bigg)=1
%\end{split}
%\end{equation*}
\end{proposition}

\begin{remark}
Therefore (E1) is not only sufficient to prove (A1b) but up to the weak assumption (T) it is also necessary.
\end{remark}

\begin{remark}
We deduce from Proposition \ref{thD} that 
\begin{equation}
\begin{split}
\Bigg\{ g^*: ~ &\mathbb{N} \to (0,1) \quad \text{such that} \quad  \sum\limits_{l \in \mathbb{N}} g^*(l) =1, \\
& \sum\limits_{l \in \mathbb{N}} g^*(l) (-\log(g^*(l))< + \infty \quad \text{and} \quad \sum\limits_{l \in \mathbb{N}} \frac{g^*(l)}{G_0(l)} < + \infty \Bigg\}
\end{split}
\end{equation}
is a subset of the Kullback-Leibler support of the Dirichlet process $DP(\alpha G_0)$.
\end{remark}

\section*{Acknowledgements}
I want to thank Elisabeth Gassiat and Judith Rousseau for their valuable comments. I also want to thank the reviewer and the editor for their helpful comments.

\appendix

\section{Proofs of key results}\label{kr}

\subsection*{Proof of Lemma \ref{KL}}\label{AKL}

For all $\theta, ~ \theta^* \in \Delta^k(\underline{q})$ the Kullback-Leibler divergence between $p^{\theta^*}_n$ and   $p_n^\theta$ verifies

\begin{equation}
\begin{split}
\frac{1}{n} & KL(p^{\theta^*}_n,p^{\theta,\mu}_n) \\
&= \frac{1}{n} \mathbb{E}_{p^{\theta^*}_n} \left(  \log\left(\frac{p^{\theta^*}_n(Y_{1:n})}{p^\theta_n(Y_{1:n})}\right)\right)\\
&= \frac{1}{n} \mathbb{E}_{p^{\theta^*}_n} \left( \log\left(\frac{\sum_{i_1,\dots, i_n=1}^k \mu^*_{i_1} Q^*_{i_1,i_2} \dots Q^*_{i_{n-1},i_n} f^*_{i_1}(Y_1) \dots f^*_{i_n}(Y_n)}
{\sum_{i_1,\dots, i_n=1}^k \mu_{i_1} Q_{i_1,i_2} \dots Q_{i_{n-1},i_n} f_{i_1}(Y_1) \dots f_{i_n}(Y_n)}\right)\right)\\
&= \frac{1}{n} \mathbb{E}_{p^{\theta^*}_n} \left( \log\left(\frac{\sum\limits_{i_1,\dots, i_n=1}^k \frac{\mu^*_{i_1} Q^*_{i_1,i_2} \dots Q^*_{i_{n-1},i_n} f^*_{i_1}(Y_1) \dots f^*_{i_n}(Y_n)}{\mu_{i_1} Q_{i_1,i_2} \dots Q_{i_{n-1},i_n} f_{i_1}(Y_1) \dots f_{i_n}(Y_n)}\mu_{i_1} Q_{i_1,i_2} \dots Q_{i_{n-1},i_n} f_{i_1}(Y_1) \dots f_{i_n}(Y_n)}
{\sum_{i_1,\dots, i_n=1}^k \mu_{i_1} Q_{i_1,i_2} \dots Q_{i_{n-1},i_n} f_{i_1}(Y_1) \dots f_{i_n}(Y_n)}\right)\right)\\
& \leq  \frac{1}{n} \mathbb{E}_{p^{\theta^*}_n} \left( \log\left(\max_{1 \leq i_1,\dots, i_n\leq k} \frac{\mu^*_{i_1} Q^*_{i_1,i_2} \dots Q^*_{i_{n-1},i_n} f^*_{i_1}(Y_1) \dots f^*_{i_n}(Y_n)}{\mu_{i_1} Q_{i_1,i_2} \dots Q_{i_{n-1},i_n} f_{i_1}(Y_1) \dots f_{i_n}(Y_n)} \right)\right)\\
& \leq \frac{1}{n} \mathbb{E}_{p^{\theta^*}_n} \left( \log\left(\max_{1 \leq i\leq k} \frac{\mu^*_{i}}{\mu_i} \left(\max_{1 \leq i,j\leq k} \frac{Q^*_{i,j}}{Q_{i,j}}\right)^{n-1} \max_{1 \leq i\leq k} \frac{f^*_{i}(Y_1)}{f_{i}(Y_1)} \dots \max_{1 \leq i\leq k} \frac{f^*_{i}(Y_n)}{f_{i}(Y_n)}  \right)\right)\\
& \leq \frac{1}{n\underline{q}} \max_{1 \leq i\leq k} \left\lvert \mu_i - \mu^*_i \right\rvert 
+   \frac{n-1}{n\underline{q}} \max_{1 \leq i,j \leq k} \left\lvert Q_{i,j} - Q^*_{i,j} \right\rvert 
+ \max_{1 \leq i\leq k} \int f^*_i(y) \max_{1 \leq j \leq k} \log \frac{f^*_j(y)}{f_j(y)} \lambda(dy).
\end{split}
\end{equation}
The last inequality comes from the following assumption:
$$\min\limits_{1\leq i,j\leq k} (\mu_i,\mu^*_i, Q_{i,j}, Q^*_{i,j})\geq \underline{q} .$$

Then for all $\epsilon>0$, for $n$ large enough,  for all  $\theta \in \Theta_\epsilon$, 
$$\frac{1}{n} KL(p^{\theta^*}_n,p^{\theta,\mu}_n) \leq \frac{3}{\underline{q}} ~ \epsilon $$

\subsection*{Proof of Theorem \ref{th1}}\label{Prth1}

This proof relies on Theorem 5 of \citet{Ba88}.
We do not assume (A2) in the first part of the proof.
First we prove that for all $a>0$,
\begin{equation}\label{eqmerge}
\mathbb{P}^{\theta^*}\left( \frac{\int_\Theta p^\theta_n (Y_1, \dots, Y_n) \pi(d\theta)}{p^{\theta^*}_n(Y_1, \dots, Y_n)}
\leq \exp(-a n) \text{ i.o.} \right)=0
\end{equation}
that is to say 
$$p^{\theta^*}_n (y_1, \dots, y_n) \lambda(dy_1) \dots \lambda(dy_n)$$
and $$\int_\Theta p^\theta_n (y_1, \dots, y_n) \lambda(dy_1) \dots \lambda(dy_n) \pi(d\theta)$$ merge with probability one.

Let $\epsilon>0$. Note that Assumption (A1a) implies that $Q^* \in \Delta^k(\underline{q})$. Then by Lemma \ref{KL}, there exists a real $\tilde{\epsilon}>0$ such that for $n$ large enough, for all $\theta \in \Theta_{\tilde{\epsilon}}$, 
 \begin{equation}\label{kleq}\frac{1}{n}KL(p^{\theta^*}_n,p^{\theta,\mu}_n) < \epsilon.\end{equation}
Moreover by Proposition 1 of \citet{Do04}, if $\theta \in \Theta(\underline{q})$ 
and if (A1c), (A1d) and (A1e) hold, 
 $$\frac{1}{n} \log \left( \frac{p^{\theta^*}_n(Y_{1:n})}{p^{\theta,\mu}_n(Y_{1:n})}\right) $$ converges $\mathbb{P}^{\theta^*}$-almost surely and in $L^1(\mathbb{P}^{\theta^*})$. 
Let $\bar{L}(\theta)$ denote this limit:
$$ \lim_{n \to \infty} \frac{1}{n} \log \left( \frac{p^{\theta^*}_n(Y_{1:n})}{p^{\theta,\mu}_n(Y_{1:n})} \right) 
 =:  \bar{L}(\theta), ~ 
\mathbb{P}^{\theta^*}\text{-a.s. and in  } L^1(\mathbb{P}^{\theta^*}).$$
Then for all $\theta \in \Theta_{\tilde{\epsilon}}$, \begin{equation}\label{klimeq}\bar{L}(\theta) \leq  \epsilon.\end{equation}
So for all $\epsilon>0$, there exists $\tilde{\epsilon}$ such that 
$$\pi\left(\theta : \bar{L}(\theta)<\epsilon\right) ~ \geq  ~ \pi(\Theta_{\tilde{\epsilon}})  ~ >0.$$
By Lemma 10 of \citet{Ba88}, for all $a>0$, (\ref{eqmerge}) is verified.

\vspace{0.5cm}

We now have to build the tests described in Theorem 5 in \citet{Ba88}, 
 to obtain posterior consistency first for the weak topology and secondly for the $D_l$-pseudometric.
 In the case of the weak topology, we follow the ideas of Section 4.4.1 in \citet{GhRa03}.
Using page 142 of \citet{GhRa03}, it is sufficient to consider 
$$ U =  \left\{P ~ : ~  \int h dP - \int h p_l^{\theta^*} d\lambda^{\otimes l}   < \epsilon,
   \right\},$$
for all   $\epsilon>0$ and  $0 \leq h \leq 1$  in the set $\mathcal{C}_b((\mathbb{R}^d)^l)$.
%Let $U$ be a weak neighborhood of $P_l^{\theta^*}$,
%then there exists $M \in \mathbb{N}$,  $\epsilon_j>0$ and  $h_j$ in the set $\mathcal{C}_b((\mathbb{R}^d)^l)$ such that
%\begin{equation*}
% \begin{split}
% U & \supset \left\{P ~ : ~ \left\lvert \int h_j dP - \int h_j p_l^{\theta^*} d\lambda^{\otimes l}  \right\rvert < \epsilon_j,
%  j=1,\dots, M \right\}  \\
%  & \supset \bigcap_{j=1}^{2M} \left\{P ~ : ~  \int \tilde{h}_j dP - \int \tilde{h}_j p_l^{\theta^*} d\lambda^{\otimes l}   < \tilde{\epsilon}_j,
%   \right\}
% \end{split}
%\end{equation*}
%with $\tilde{h}_j = (h_j - \inf h_j)/(\sup h_j - \inf h_j)$  for $1 \leq j \leq M$,
%$\tilde{h}_{M+j}=1- \tilde{h}_j$ for $1 \leq j \leq M$  and $\tilde{\epsilon_j}= \epsilon_j/(\sup h_j - \inf h_j)$.
%For all $1 \leq j \leq 2M$,
%\begin{equation*}
% \alpha_j 
%= \int h_j p_l^{\theta^*} d\lambda^{\otimes l}
% < \inf_{P ~: ~ \int h_j dP - \int h_j p_l^{\theta^*} d\lambda^{\otimes l}   \geq \epsilon_j} \int h_j dP
%= \gamma_j.
%\end{equation*}
%For $n$ a multiple of $l$ (one may generalise the following when $n$ is not a multiple of $l$),
%denote 
Choosing $\alpha$ and $\gamma$ as in  page 128 of  \citet{GhRa03}, if
$$S^n=\left\{ y_1, \dots, y_n ~: ~ \frac{l}{n} \sum_{j=0}^{n/l-1} h(y_{jl+1}, \dots, y_{jl+ l}) > \frac{\alpha + \gamma}{2} \right\},$$
then
\begin{equation}\label{plip}
\begin{split}
 P^{\theta^*} (S^n) & = P^{\theta^*}\left\{  \sum_{j=0}^{n/l-1} \left (h(y_{jl+1}, \dots, y_{jl+ l})  - \int h p_l^{\theta^*} d\lambda^{\otimes l}\right)>\frac{n}{l} \frac{\gamma - \alpha }{2} \right\} \\
 &\leq  \exp \left(- \frac{n (\gamma - \alpha ) ^2 (\min_{i,j}Q^*_{i,j})^2 }{2 l (2-k\min_{i,j}Q^*_{i,j})^2 } \right)
 \end{split}
\end{equation}
and for all $\theta \in \Theta(\underline{q})$ such that $\int h dP^\theta - \int h p_l^{\theta^*} d\lambda^{\otimes l}   \geq \epsilon$,
\begin{equation}\label{plop}
 \begin{split}
   P^\theta((S^n)^c)
   &\leq P^{\theta}\left\{   \sum_{j=0}^{n/l-1} \left (- h(y_{jl+1}, \dots, y_{jl+ l})  + \int h p_l^{\theta} d\lambda^{\otimes l}\right) \geq \frac{n}{l} \frac{\gamma - \alpha }{2} \right\} \\
  & \leq  \exp \left(- \frac{n (\gamma - \alpha ) ^2 (\min_{i,j}Q_{i,j})^2 }{2 l (2-k\min_{i,j}Q_{i,j})^2 } \right) \leq  \exp \left(- \frac{n (\gamma - \alpha ) ^2 \underline{q}^2 }{2 l } \right),
 \end{split}
\end{equation}
using  the upper bound from the proof of Theorem 4 of \citet{GaRo12} based on Corollary 1 in \citet{Ri00}.

Using Theorem 5 of \citet{Ba88} and combining Equations \eqref{plip} and \eqref{plop},
\begin{equation*}
\begin{split}
 &P^{\theta^*}\Bigg( \pi \bigg(\Big\{\theta : \int h dP^\theta - \int h p_l^{\theta^*} d\lambda^{\otimes l}  < \epsilon \Big\}^c ~ \bigg|  ~ Y_{1:n}\bigg)\geq e^{-nr}  \text{, i.o. }  \Bigg)=0
 \end{split}
\end{equation*}
which implies that for all weak neighborhood $U$ of $P_l^{\theta^*}$,
\begin{equation*}
  P^{\theta^*}(\left( \pi(U^c|Y_{1:n})\geq \exp(-nr) \text{ i.o. }  \right)=0,
\end{equation*}
so that 
$$\mathbb{P}^{\theta^*}\left(\lim_{n \to \infty} \pi(U|Y_{1:n})=1 \right)=1 .$$

\vspace{0.5cm}

We now assume (A2) and obtain consistency for the $D_l$-pseudometric.
Let $\epsilon>0$ and let
$$U=\left\{ \theta ~ : ~D_l(\theta,\theta^*)<  \frac{2 \epsilon}{k \underline{q}}   \right\} \supset \left\{ \theta ~ : ~D_l(\theta,\theta^*)<    
 \epsilon \frac{2 - k \min_{1\leq i,j\leq k} Q_{i,j}  }{k  \min_{1\leq i,j\leq k} Q_{i,j} } \right\},$$
 be a $D_l$-neighborhood of $\theta^*$.
%Then there exists a real $\delta>0$ such that 
%$$\left\{ \theta ~ : ~ D_l(\theta,\theta^*) <    
%\frac{2 \delta }{k  \underline{q}} \right\} \subset U.$$
Let 
 $$B_n^c=\Delta^k(\underline{q}) \times \mathcal{F}_n,$$
%and
%$$C_n=B_n^c\cap A^c,$$
%so that 
%\begin{equation}\label{eqU}
%A\cup B_n \cup C_n=\Theta.
%\end{equation}
%Then by assumption (A2), for all $n\in \mathbb{N}$,
so that
\begin{equation}\label{eqB}
\pi(B_n)= \pi_f({\mathcal{F}_n}^c)\lesssim \exp(-n r_1).
\end{equation}

In the proof of Theorem 4 of \citet{GaRo12},
%Let $N=N(\frac{\delta}{12}, C_n, D_l(\cdot,\cdot))$ and $\theta_1, \dots \theta_{N} \in C_n$ 
%such that  for all $\theta \in C_n$, there exists $j \leq N$ such that
%$D_l(\theta, \theta_j)  \leq \frac{\delta}{12}$.
%For all $j \leq N$, 
%let
%$$t_j = \frac{n}{4l}    D_l(\theta_j,\theta^*) $$
%$$A_j = \left\{ \left( y_1, \dots, y_l \right) ~ : ~  
%p_l^{\theta^*}(y_1, \dots, y_l) \leq p_l^{\theta_j}(y_1, \dots, y_l) \right\} , $$
%$$ T_j = \left\{ 
%		\left(y_1, \dots, y_n  \right) : ~ 
% 		 \sum_{i=1}^{n/l} 
%		\left(              \mathds{1}_{     \left(y_{li-l+1}, \dots , y_{li}\right)        \in A_j}
%			-P^{\theta^*}  \left(              \left(Y_{1}, \dots , Y_{l}\right)   \in A_j      \right) 
%		\right)  >t_j 
%	\right\}$$
%and 
%$$S_n = \cup_{1 \leq j \leq N} T_j \subset (\mathbb{R}^d)^n.$$
%In \cite{GaRo12}
 it is proved that for all $n$ large enough, there exists a test $\psi_n$ such that
 \begin{equation}\label{eqth1}
 \begin{split}
 \mathbb{E}^{\theta^*}(\psi_n) &\leq
 N\left(\frac{\epsilon}{12}, \Delta^k(\underline{q}) \times \mathcal{F}_n, D_l\right) \exp\left(-\frac{n \epsilon^2}{8l} \frac{ k^2 (\min_{i,j}Q^*_{i,j})^2}{(2-k\min_{i,j}Q^*_{i,j})^2}\right)\\
  &\leq
 N\left(\frac{\epsilon}{12}, \Delta^k(\underline{q}) \times \mathcal{F}_n, D_l\right) \exp\left(-\frac{n \epsilon^2 k^2 \underline{q}^2}{32 l}\right)
 \end{split}
 \end{equation}
\begin{equation}\label{eqtest1}
 \sup_{\theta \in  U^c \cap B_n^c} \mathbb{P}^{\theta,\mu} (1-\psi_n) \leq \exp\left( -\frac{n \epsilon^2}{32l}  \right) .
 \end{equation}
% where we have used the computations from the proof of Theorem 4 of \citet{GaRo12} with
% $\rho_\theta \text{ (in \cite{GaRo12}) }= (1-k\min_{i,j} Q_{i,j})^{-1} $, 
% $R_\theta \text{ (in \cite{GaRo12}) }=1$,
% $A_n\text{ (in \cite{GaRo12}) }=A^c$, 
% $\epsilon_n \text{ (in \cite{GaRo12}) }=\delta$ 
% and $\delta \text{ (in \cite{GaRo12}) }= \frac{\epsilon_n}{12} \text{ (in \cite{GaRo12}) }= \frac{\delta}{12}$.

 Note that for all $\theta, \tilde{\theta}$ in $\Theta(\underline{q})$,
 \begin{equation*}
  D_l(\theta, \tilde{\theta}) \leq \lVert \mu^\theta - \mu^{\tilde{\theta}} \rVert_1 + k(l-1) \lVert Q-\tilde{Q} \rVert
   + l \max_{1 \leq j\leq k} \lVert f_j - \tilde{f_j} \rVert_{L_1(\lambda)}
 \end{equation*}
 The function $Q \to \mu^{Q}$ is continuous on the compact $\Delta^k(\underline{q})$ and 
thus 
 is uniformly continuous: there exists $\alpha>0$ such that for all 
 $\theta, \tilde{\theta}$ in $\Theta(\underline{q})$ such that $\lVert Q - \tilde{Q} \rVert <\alpha$
 then $\lVert \mu^\theta -\mu^{\tilde{\theta}} \rVert_1 < \frac{\epsilon}{36}$.
 This implies that
 \begin{equation}\label{eqth11}
 \begin{split}
  N & \left(\frac{\epsilon}{12}, \Delta^k(\underline{q}) \times \mathcal{F}_n, D_l\right) \\
  & \leq N\left(\min\bigg(\frac{\epsilon}{36k(l-1)},\alpha\bigg), \Delta^k(\underline{q}) , \lVert \cdot \rVert \right)
     N\left(\frac{\epsilon}{36l},  \mathcal{F}_n, d(\cdot,\cdot) \right)\\
  & \leq \left(\max \left(\frac{36k(l-1)}{\epsilon},\frac{1}{\alpha}\right) \right)^{k(k-1)} 
      N\left(\frac{\epsilon}{36l},  \mathcal{F}_n, d(\cdot,\cdot) \right)
  \end{split}
 \end{equation}
 Then combining Equations   
%and Assumption (A2), using Borel-Cantelli lemma we get
% \begin{equation}\label{eqtest2}
% \mathbb{P}^{\theta^*}(S_n \text{ infinitely often })=0.  
% \end{equation}
%
% 
% 
%Thus combining Equations (\ref{eqmerge}), 
(\ref{eqB}), (\ref{eqth1}), (\ref{eqtest1}), (\ref{eqth11}) and using Theorem 5 of \citet{Ba88}, 
there exists $r>0$ such that
\begin{equation}\label{eqBa}
\mathbb{P}^{\theta^*}\Bigg( \pi  \left(U^c|Y_{1:n}\right)  
 \geq \exp(-nr) \text{ i.o. }  \Bigg) =0  .
\end{equation}
%As 
%\begin{equation*}
%A \subset 
% \left\{ \theta ~ : ~D_l(\theta,\theta^*) <    
%\frac{2 \delta }{k  \underline{q}} \right\} 
% \subset U ,
%\end{equation*}
And Equation (\ref{eqBa}) implies that for all $\epsilon>0$,
%$$\mathbb{P}^{\theta^*}\left( \pi\left(~ \  |Y_{1:n}\right)\geq \exp(-nr) \text{ i.o. }  \right)= 0.$$
$$\mathbb{P}^{\theta^*}\left(\lim_{n \to \infty} \pi(~ \left\{ \theta: ~ D_l(\theta,\theta^*)< \epsilon  \right\} ~ | ~ Y_{1:n})=1 \right)=1. $$

%%%%%%%%%%%%%%%%%%%%%%%%%%%%%%%%%%%%%%%%%%%%%%%%%%%%

\subsection*{Proof of Theorem \ref{th2}}\label{P:th2}
Using Theorem \ref{th1}, it is sufficient to show that for all weak neighborhood $U_{f^*}$ of $f^* \lambda$ 
and neighborhood $U_{Q^*}$ of $Q^*$,
there exists a $D_3$-neighborhood $U_{\theta^*}$ of $\theta^*$ such that 
\begin{equation}\label{eqsp}
U_{\theta^*} \subset
\left\{ \exists \sigma \in \mathcal{S}_k; ~ \sigma Q \in U_{Q^*}, ~ f_{\sigma(i)} \in U_{f^*_i},  ~  i=1\dots k \right\}   .
\end{equation}
Following \citet{GaClRo13}, it is equivalent to show that for all sequences $\theta^n$ in $ \Theta(\underline{q})$
such that $D_3(\theta^n,\theta^*) \to 0$,
there exists a subsequence, that we denote again $\theta^n$, of $\theta^n$ and $\bar{\theta} \in \Theta$
such that $\lVert Q^n-\bar{Q}\rVert \to 0$, $f^n_i \lambda$ tends to $\bar{f}_i \lambda$ in the weak topology on probabilities 
for all $i\leq k$ and
$p_3^{(Q^*,f^*)} = p_3^{(\bar{Q},\bar{f})}$.

Let  $\theta^n$ in $ \Theta(\underline{q})$
such that $D_3(\theta^n,\theta^*) \to 0$.
As $\Delta^k(\underline{q})$ is a compact set, 
there exists a subsequence of $Q^{n}$ that we denote again
$Q^{n}$ which tends to $\bar{Q} \in \Delta^k(\underline{q})$. 
Writing $\mu^n$ the (sub)sequence of the stationary distribution associated to $Q_n$, then 
$\mu^{n} \to \bar{\mu}$ where $\bar{\mu}$ is the stationary distribution associated to $\bar{Q}$.
Moreover,
\begin{equation*}
\begin{split}
D_3( & \theta^{n}, \theta^*)= \lVert p_3^{\theta_n} - p_3 ^{\theta^*} \rVert_{L_1(\lambda^{\otimes 3})} \\
& \geq  \int \Big\lvert \sum_{1 \leq i_1, i_2, i_3 \leq k }\mu_{i_1}^{n} Q_{i_1,i_2}^{n} Q_{i_2,i_3}^{n} 
       f_{i_1}^{n}(y_1) f_{i_2}^{n}(y_2) f_{i_3}^{n}(y_3) -  \\
& \quad \mu_{i_1}^* Q_{i_1,i_2}^* Q_{i_2,i_3}^* 
        f_{i_1}^*(y_1) f_{i_2}^*(y_2) f_{i_3}^*(y_3) \Big\rvert ~ \lambda(dy_1)\lambda(dy_2)\lambda(dy_3) \\
& \geq - \sum_{1 \leq i_1, i_2, i_3 \leq k }\left \lvert \mu_{i_1}^{n} Q_{i_1,i_2}^{n} Q_{i_2,i_3}^{n} - 
            \bar{\mu}_{i_1} \bar{Q}_{i_1,i_2} \bar{Q}_{i_2,i_3} \right\rvert  + \\
& \quad  \int \Big\lvert \sum_{1 \leq i_1, i_2, i_3 \leq k } \bar{\mu}_{i_1} \bar{Q}_{i_1,i_2} \bar{Q}_{i_2,i_3} 
          f_{i_1}^{n}(y_1) f_{i_2}^{n}(y_2) f_{i_3}^{n}(y_3)- \\
& \qquad   \mu_{i_1}^* Q_{i_1,i_2}^* Q_{i_2,i_3}^* 
               f_{i_1}^*(y_1) f_{i_2}^*(y_2) f_{i_3}^*(y_3) \Big\rvert ~ \lambda(dy_1)\lambda(dy_2)\lambda(dy_3)
\end{split}
\end{equation*}

Since 
$\displaystyle{ \sum_{1 \leq i_1, i_2, i_3 \leq k }\Big\lvert \mu_{i_1}^{n} Q_{i_1,i_2}^{n} Q_{i_2,i_3}^{n} 
- \bar{\mu}_{i_1} \bar{Q}_{i_1,i_2} \bar{Q}_{i_2,i_3} \Big\rvert}$
tends to zero,
\begin{equation}\label{semconv}
\begin{split}
\lim_n &  \int \Big\lvert \sum_{1 \leq i_1, i_2, i_3 \leq k } \bar{\mu}_{i_1} \bar{Q}_{i_1,i_2} \bar{Q}_{i_2,i_3}    f_{i_1}^{n}(y_1) f_{i_2}^{n}(y_2) f_{i_3}^{n}(y_3)- \\
& \qquad \mu_{i_1}^* Q_{i_1,i_2}^* Q_{i_2,i_3}^* f_{i_1}^*(y_1) f_{i_2}^*(y_2) f_{i_3}^*(y_3) \Big\rvert ~ \lambda(dy_1)\lambda(dy_2)\lambda(dy_3) =0
\end{split}
\end{equation}

 Let $F^n_1, \dots, F^n_k$ be the probability distribution with respective densities $f^n_1, \dots, f^n_k$
 with respect to $\lambda$.
 Since $$\sum_{i_1,i_2,i_3} \bar{\mu}_{i_1} \bar{Q}_{i_1,i_2} \bar{Q}_{i_2,i_3}
F^n_{i_1} \otimes F^n_{i_2} \otimes F^n_{i_3}$$ 
converges in total variation, it is tight and for all $1 \leq i\leq k$,  $(F_i^n)_n$ is tight.
By Prohorov's theorem, for all $1 \leq i \leq k$ there exists a subsequence denoted
 $F^n_i$ of $F^n_i$
 which weakly converges to $\bar{F}_i$.
This in turns implies that 
$$\sum_{i_1,i_2,i_3} \bar{\mu}_{i_1} \bar{Q}_{i_1,i_2} \bar{Q}_{i_2,i_3}
F^n_{i_1} \otimes F^n_{i_2} \otimes F^n_{i_3}$$
 weakly converges to 
$$\sum_{i_1,i_2,i_3} \bar{\mu}_{i_1} \bar{Q}_{i_1,i_2} \bar{Q}_{i_2,i_3}
\bar{F}_{i_1} \otimes \bar{F}_{i_2} \otimes \bar{F}_{i_3},$$
which combined with \eqref{semconv},  leads to
\begin{equation*}
\begin{split}
\sum_{i_1,i_2,i_3} & \bar{\mu}_{i_1} \bar{Q}_{i_1,i_2} \bar{Q}_{i_2,i_3}
\bar{F}_{i_1} \otimes \bar{F}_{i_2} \otimes \bar{F}_{i_3}\\
&= \sum_{ i_1, i_2, i_3  } \mu_{i_1}^* Q_{i_1,i_2}^* Q_{i_2,i_3}^* f_{i_1}^*\lambda \otimes f_{i_2}^*\lambda \otimes f_{i_3}^*\lambda
\end{split}
 \end{equation*}

By \citet{GaClRo13}, $\bar{Q}=Q^*$, so $\bar{\mu}=\mu^*$ and 
$\bar{F}_i=f^*_i \lambda$ up to a label swapping, that is there exists a permutation $\sigma \in \mathcal{S}_k$ such that 
$\sigma\bar{Q}= Q^*$ and $\bar{F}_{\sigma(i)}= f^*_{i} \lambda$ so that Equation \eqref{eqsp} holds.

% \subsection{Tests}
% \begin{proposition}\label{test}
% (\citet{GaRo12})
% 
% Let $n>0$, $\delta>0$ and  $A_n = \left\{ \theta ~ : ~ \lVert p_{l}^\theta - p_{l}^{\theta^*}\rVert_1 <    
% 4 l \delta \frac{4- 3 k \min_{1\leq i,j\leq k} Q_{i,j}  }{k  \min_{1\leq i,j\leq k} Q_{i,j} } \right\}$.
%  If $Q^* \in \Delta^k(\underline{q})$, then for all set $B_n \subset \Theta$ such that $B_n^c \cap A_n^c \neq \emptyset $, there exists a measurable set $S_n \subset \mathcal{Y}^n$ such that 
%  $$\mathbb{P}^*(S_n) \lesssim N(\delta, B_n^c , D_l) \exp\left(-\frac{n \delta^2}{4}\right)$$
%  $$\sup_{\theta \in B_n^c \cap A_n^c} \mathbb{P}^\theta (S_n) \lesssim \exp\left( -n \delta^2  \right) $$
% \end{proposition}

%%%%%%%%%%%%%%%%%%%%%%%%%%%%%%%%%%%%%%%%%%%%%%%%%%%%%

\subsection*{Proof of Theorem \ref{FD}}\label{P:FD}
To prove Theorem \ref{FD} we need the following lemma:
\begin{lemma}\label{lem}
Let $\epsilon>0$, for all $0<\epsilon_1<1$, $N>0$, $1 \leq j < N$ and $c>0$ such that 
$$0<\frac{\epsilon_1 2^{2N}k^N}{c(c-\epsilon_1)} < \frac{\epsilon}{3} 
\text{ and }
\frac{2 (1-\underline{q})^{N+1-j}}{\underline{q} + (1-\underline{q})^{N+1-j}}  < \frac{\epsilon}{3}.$$
If $p_N^{\theta^*}(Y_{1:N})>c$, then for all $1\leq l \leq k$ and for all $n>N$,
\begin{equation*}
\begin{split}
\bigg\{ \theta \in \Delta^k(\underline{q}) : \lVert p_N^{\theta^*}  - p_N ^{\theta} \rVert_{l_1}< \epsilon_1 , ~  \exists \sigma \in \mathcal{S}_k, ~ \lvert\mu^\theta_{\sigma(i)}-\mu^*_i\rvert < \epsilon_1 , ~  \lVert \sigma Q - Q^* \rVert < \epsilon_1 , \\  
\max_{1\leq i \leq k}\lVert f_{\sigma(i)}-f^*_i \rVert_{l_1}< \epsilon_1  \bigg\} \\
\subset \left\{ \theta  \in \Delta^k(\underline{q}) : ~ \lvert P^{\theta^*}(X_j=l ~ | ~ Y_{1:n})  - P^{\theta}(X_j=l ~ | ~ Y_{1:n}) \rvert ~ < ~ \epsilon \right\}
\end{split}
\end{equation*}
\end{lemma}

\begin{proof}[Proof of Lemma \ref{lem}]
Let $\theta  \in \Delta^k(\underline{q})$ be such that  $\lVert p_N^{\theta^*}  - p_N ^{\theta} \rVert_{l_1}< \epsilon_1$  and there exists $\sigma \in \mathcal{S}_k$ such that  $\max_{1\leq i \leq k } \lvert \mu^\theta_{\sigma(i)}-\mu^*_i\rvert < \epsilon_1$,  $\lVert \sigma Q -  Q^* \rVert < \epsilon_1$ and
 $\max_{1\leq i \leq k}\lVert f_{\sigma(i)}-f^*_{i} \rVert_{l_1}< \epsilon_1 $.

To bound $\lvert P^{\theta^*}(X_j=l ~ | ~ Y_{1:n})  - P^{\theta}(X_j=l ~ | ~ Y_{1:n}) \rvert $, we now prove that it is sufficient to bound $\lvert P^{\theta^*}(X_j=l ~ | ~ Y_{1:N})  - P^{\theta}(X_j=l ~ | ~ Y_{1:N}) \rvert $ with $N<n$ a well chosen fixed integer thanks to the exponential forgetting of the HMM. Let $1 \leq a \leq k$,

\begin{equation}\label{lem1}
\begin{split}
\lvert & P^{\theta^*}(X_j=l ~ | ~ Y_{1:n})  - P^{\theta}(X_j=l ~ | ~ Y_{1:n}) \rvert  \\
& \leq A_{\theta^*} + \lvert P^{\theta^*}(X_j=l ~ | ~ Y_{1:N})  - P^{\theta}(X_j=l ~ | ~ Y_{1:N}) \rvert + A_{\theta},
\end{split}
\end{equation}

where for $\tilde{\theta}\in\{\theta, \theta^*\}$,
\begin{equation*}
\begin{split} A_{\tilde{\theta}} =
&  \bigg\lvert \frac{P^{\tilde{\theta}}(Y_{1:N},X_j=l)   \sum\limits_{1 \leq b \leq k} P^{\tilde{\theta}} (Y_{N+1:n}~ | X_{N+1}=b) P^{\tilde{\theta}}( X_{N+1}=b | X_j=l,Y_{j:N} )}{\sum\limits_{1\leq m \leq k} P^{\tilde{\theta}}(Y_{1:N},X_j=m)   \sum\limits_{1 \leq b \leq k} P^{\tilde{\theta}} (Y_{N+1:n}~ | X_{N+1}=b) P^{\tilde{\theta}}( X_{N+1}=b | X_j=m , Y_{j:N})} -  \\
&  \frac{P^{\tilde{\theta}}(Y_{1:N},X_j=l)   \sum\limits_{1 \leq b \leq k} P^{\tilde{\theta}} (Y_{N+1:n}~ | X_{N+1}=b) P^{\tilde{\theta}}( X_{N+1}=b | X_j=a, Y_{j:N} )}{\sum\limits_{1\leq m \leq k} P^{\tilde{\theta}}(Y_{1:N},X_j=m)   \sum\limits_{1 \leq b \leq k} P^{\tilde{\theta}} (Y_{N+1:n}~ | X_{N+1}=b) P^{\tilde{\theta}}( X_{N+1}=b | X_j=a , Y_{j:N} )}  
 \bigg\rvert.
\end{split}
\end{equation*}

Using Corollary 1 of \citet{Do04}, i.e. the exponential forgetting of the HMM, we obtain for all $(\omega,m) \in \{1, \dots, k\}^2$,
\begin{equation*}
\begin{split}
& \left \lvert P^{\tilde{\theta}}( X_{N+1}=b | X_j=m , Y_{j:N} ) -  P^{\tilde{\theta}}( X_{N+1}=b | X_j=\omega , Y_{j:N} ) \right \rvert \\
& ~ \leq  (1-\underline{q})^{N+1-j}   \leq 
(1-\underline{q})^{N+1-j}  \frac{ P^{\tilde{\theta}}( X_{N+1}=b | X_j=\omega , Y_{j:N} )}{\underline{q}}
\end{split}
\end{equation*}
 so that for  $\tilde{\theta}\in\{\theta, \theta^*\}$,  
\begin{equation}
\label{lem2}A_{\tilde{\theta}}  \leq \frac{2 (1-\underline{q})^{N+1-j}}{\underline{q} + (1-\underline{q})^{N+1-j}}.
\end{equation}

Moreover,
\begin{equation*}
\begin{split}
& P^{\theta^*}(X_j=l ~ | ~ Y_{1:N})  - P^{\theta}(X_j=l ~ | ~ Y_{1:N})  \\
& =  \frac{\sum\limits_{a_{1:j-1},a_{j+1:N}}\mu^*_{a_1}Q^*_{a_1,a_2}  \dots Q^*_{a_{j-1},l}Q^*_{l,a_{j+1}} \dots Q^*_{a_{N-1},a_N}f^*_{a_1}(Y_{a_1})\dots f^*_l(Y_j)\dots f^*_{a_N}(Y_N)}{p^{\theta^*}_N(Y_{1:N})}   \\
& - \frac{\sum\limits_{a_{1:j-1},a_{j+1:N}}\mu^{}_{a_1}Q_{a_1,a_2}  \dots Q_{a_{j-1},l}Q_{l,a_{j+1}} \dots Q_{a_{N-1},a_N}f_{a_1}(Y_{a_1})\dots f_l(Y_j)\dots f_{a_N}(Y_N)}{p^{\theta}_N(Y_{1:N})}     \\
& \leq \frac{(1+\epsilon_1/c)  \sum\limits_{a_{1:j-1},a_{j+1:N}}\mu^*_{a_1} \dots f^*_{a_N}(Y_N)
-\sum\limits_{a_{1:j-1},a_{j+1:N}}\mu^{}_{a_1}\dots f_{a_N}(Y_N)}{(1+\epsilon_1/c)p_N^{\theta^*}(Y_{1:N})}\\
& \leq \frac{(1+\epsilon_1/c)  \sum\limits_{a_{1:j-1},a_{j+1:N}}\mu^*_{a_1} \dots f^*_{a_N}(Y_N)
-\sum\limits_{a_{1:j-1},a_{j+1:N}}(\mu^{}_{a_1}-\epsilon_1)\dots (f_{a_N}(Y_N)-\epsilon_1)}{c+\epsilon_1}   \\
& \leq  \frac{\max(\epsilon_1,\epsilon_1/c) \sum\limits_{a_{1:j-1},a_{j+1:N}} 2^{2N}  }{c+ \epsilon_1}
\leq \frac{\epsilon_1 2^{2N}k^N}{c(c+\epsilon_1)}.
\end{split}
\end{equation*}
Similarly
\begin{equation*}
P^{\theta}(X_j=l ~ | ~ Y_{1:N})  - P^{\theta^*}(X_j=l ~ | ~ Y_{1:N}) \leq \frac{\epsilon_1 2^{2N}k^N}{c(c-\epsilon_1)}
\end{equation*}
so that 
\begin{equation}\label{lem3}
 \big\lvert P^{\theta^*}(X_j=l ~ | ~ Y_{1:N})  - P^{\theta}(X_j=l ~ | ~ Y_{1:N}) \big\rvert \leq \frac{\epsilon_1 2^{2N}k^N}{c(c-\epsilon_1)}.
\end{equation}
Combining Equations (\ref{lem1}), (\ref{lem2}) and \eqref{lem3}, we obtain
\begin{equation*}
\begin{split}
\lvert  &P^{\theta^*}(X_j=l ~ | ~ Y_{1:n})  - P^{\theta}(X_j=l ~ | ~ Y_{1:n}) \rvert  \\
&\leq 2 \frac{2 (1-\underline{q})^{N+1-j}}{\underline{q} + (1-\underline{q})^{N+1-j}} + \frac{\epsilon_1 2^{2N}k^N}{c(c-\epsilon_1)} < \epsilon. 
\end{split}
\end{equation*}
\end{proof}

We prove Theorem \ref{FD} for $m=1$, one may easily generalizes the proof.
Let $\beta>0$, $j>0$ and $\epsilon>0$, we fix $N$ and $c>0$ such that
 $$\frac{2 (1-\underline{q})^{N+1-j}}{\underline{q} + (1-\underline{q})^{N+1-j}}<\frac{\epsilon}{3}
 \text{ and }
 P^{\theta^*}\big(p_N^{\theta^*}(Y_{1:N})>c\big)>\sqrt{1-\beta}$$
 then we choose $\epsilon_1$ such that
$$0< \frac{\epsilon_1 2^{2N}k^N}{c(c-\epsilon_1)} < \frac{\epsilon}{3}.$$

Posterior consistency for the marginal distribution in $l_1$ and for all components of the parameter i.e. Theorems \ref{th1} and \ref{th2} imply that there exists $M$ such that $P^{\theta^*}$-a.s., for all $n \geq M$,
$$\pi\left( \{ \theta ~: ~ D_N(\theta,\theta^*) \}<\epsilon_1   ~ \big| ~ Y_{1:n} \right)>\frac{\sqrt{1-\beta}+1}{2} $$
and
\begin{equation*}
\begin{split}
 \pi\bigg( \{ \theta ~: ~ \exists \sigma \in \mathcal{S}_k, ~ \max_{1\leq i\leq k} \lvert \mu_{\sigma(i)} &- \mu^*_i \rvert < \epsilon_1, ~  \lVert \sigma Q - Q^*\rVert<\epsilon_1, ~\\
 &\max_{1 \leq i \leq k} \lVert f_{\sigma(i)} -f^*_i \rVert_{l_1} \}<\epsilon_1 ~ \bigg| ~ Y_{1:n} \bigg)>\frac{\sqrt{1-\beta}+1}{2} 
\end{split}
\end{equation*}
 so that for all $n\geq \max(N,M)$,
\begin{equation*}
\begin{split}
\mathbb{E}^{\theta^*} & \left( \pi \bigg(\left\lvert P^\theta(X_j=l ~|Y_{1:n}) - P^{\theta^*}(X_j=l ~|~Y_{1:n})\right\rvert <\epsilon   |Y_{1:n}\bigg)   \right) \\
& \geq \mathbb{E}^{\theta^*}  \left( \mathds{1}_{p^{\theta^*}_N(Y_{1:N})>c} \pi \bigg( \left\lvert P^\theta(X_j=l ~|Y_{1:n}) - P^{\theta^*}(X_j=l ~|~Y_{1:n})\right\rvert <\epsilon   |Y_{1:n}\bigg)   \right) \\
& \geq 1 -\beta.
\end{split}
\end{equation*}
Then for all $\alpha>0$,
\begin{equation*}
\begin{split}
P^{\theta^*} &\left( \pi\bigg(
\left\lvert  P^\theta(X_j=l ~|Y_{1:n}) - P^{\theta^*}(X_j=l ~|~Y_{1:n})\right\rvert <\epsilon   |Y_{1:n}\bigg)  < 1 -\alpha \right) \\
& \leq \frac{1}{\alpha} \left( 1- \mathbb{E}^* \left(  \pi\bigg(
\left\lvert  P^\theta(X_j=l ~|Y_{1:n}) - P^{\theta^*}(X_j=l ~|~Y_{1:n})\right\rvert <\epsilon   |Y_{1:n}\bigg)  \right)\right) \\
& \to 0.
\end{split}
\end{equation*}

%%%%%%%%%%%%%%%%%%%%%%%%%%%%%%%%%%%%%%%%%%%%
%%%%%%%%%%%%%%%%%%%%%%%%%%%%%%%%%%%%%%%%%%%%%%%%%%%%%%%%%%%%%%%%%%%%%%%%%%%%%%%%%%%%%%
%%%%%%%%%%%%%%%%%%%%%%%%%%%%%%%%%%%%%%%%%%%

\subsection*{Proof of Proposition \ref{thD}}\label{P:thD}
Note that for all $1 \leq  i \leq k$,
\begin{equation*}
\begin{split}
\int_{\mathcal{F}^k} & \sum_{l=1}^{+\infty} f^*_i(l) \max_{1\leq j \leq k} \left( -\log(f_j(l)) \right) (DP(\alpha G_0))^{\otimes k}(df) \\
&\leq  \sum_{l=1}^{+\infty} f^*_i(l) \sum_{1\leq j \leq k} \int_{\mathcal{F}^k}\left( -\log(f_j(l)) \right) (DP(\alpha G_0))^{\otimes k}(df)\\
& \lesssim \sum _{l=1}^{+\infty} \frac{f^*_i(l)}{\alpha G_0(l)}
 \end{split}
\end{equation*}
so that using Assumption (E1), 
\begin{equation*}
\begin{split}
\big(DP(\alpha G_0)\big)^{\otimes k} & \bigg(f_1, \dots ,f_k ~ : ~ \forall 1\leq i \leq k, \\
 & \qquad \sum_{l=1}^{+\infty} f_i^*(l)\max_{1 \leq j \leq k} (-\log(f_j(l)))<+\infty \bigg)=1.
\end{split}
\end{equation*}
Note that for all $\epsilon>0$, 
\begin{equation*}
\begin{split}
 &  \left\{f_1, \dots ,f_k ~ : ~ \forall 1\leq i \leq k,  ~ \sum_{l=1}^{+\infty} f_i^*(l)\max_{1 \leq j \leq k} (-\log(f_j(l)))<+\infty \right\} \\
& \subset  
\bigcup_{N \in \mathbb{N}}  \left\{ f_1, \dots ,f_k ~ : ~ \forall 1\leq i \leq k,  ~ \sum_{l=N}^{+\infty} f_i^*(l)\max_{1 \leq j \leq k} (-\log(f_j(l)))< \epsilon \right\},
\end{split}
\end{equation*}
thus arguing by contradiction, for all $\epsilon>0$,
 there exists $L_\epsilon$ such that 
 \begin{equation*}
 \begin{split}
 \big(DP(\alpha G_0)\big)^{\otimes k} & \bigg(f_1, \dots ,f_k ~ : ~ \forall 1\leq i \leq k, \\
 &\qquad \sum_{l>L_\epsilon} f_i^*(l)\max_{1 \leq j \leq k} (-\log(f_j(l)))<\epsilon \bigg)>0.
 \end{split}
 \end{equation*}
 
 Using the tail free property of the Dirichlet process, for all $ 1 \leq j \leq k$,
 $$ \sum_{l>L_\epsilon} f_i^*(l)\max_{1 \leq j \leq k} (-\log(f_j(l)))<\epsilon  $$
   and
   \begin{equation}\label{eqind}
\left(\frac{f_j(1)}{f_j(l \leq L_\epsilon)}, \dots, \frac{f_j(L_\epsilon)}{f_j(l \leq L_\epsilon)} \right)
\end{equation}
are independent given $f_j(l>L_\epsilon)$ and
\eqref{eqind} given $f_j(l>L_\epsilon)$ has a Dirichlet distribution with parameter $(\alpha G_0(1), \dots, \alpha G_0(L_\epsilon))$.
 Then for all $\epsilon>0$, there exists $L_\epsilon$ such that for all $\delta \in (0,1)$,
 \begin{equation}\label{eqE4}
 \begin{split}
  \big(DP(\alpha G_0)\big)^{\otimes k} & \bigg(f_1, \dots ,f_k ~ : ~ \forall 1 \leq i \leq k,
   ~  \sum_{l>L_\epsilon} f_i^*(l)\max_{1 \leq j \leq k} (-\log(f_j(l)))<\frac{\epsilon}{2},\\ 
 & \qquad  \forall l \leq L_\epsilon, 
  ~ \lvert f_j(l)- f^*_j(l) \rvert \leq c\delta \bigg) >0
  \end{split}
  \end{equation}
  where $ c = \min_{1 \leq i \leq k} \min_{l \leq L_\epsilon, f^*_i(l)>0} f^*_i(l)$.
  
  For all $f_1, \dots, f_k$ such that for all $1 \leq i \leq k$,
  $$\sum_{l>L_\epsilon} f_i^*(l)\max_{1 \leq j \leq k} (-\log(f_j(l)))<\frac{\epsilon}{2}$$
  and for all $l \leq L_\epsilon$, 
  $ \lvert f_i(l)- f^*_i(l) \rvert \leq c\delta $,
(A1e) holds and
  \begin{equation}\label{eqE3}
  \begin{split}
  \sum_{l \in \mathbb{N}} & f^*_i(l) \max_{1\leq j \leq k}\log\left(\frac{f^*_j(l)}{f_j(l)} \right) \\
  & = \sum_{l\leq L_\epsilon} f^*_i(l) \max_{1\leq j \leq k}\log\left(\frac{f^*_j(l)}{f_j(l)} \right)
  + \sum_{l > L_\epsilon} f^*_i(l) \max_{1\leq j \leq k}\log(f^*_j(l)) \\
 & \quad + \sum_{l > L_\epsilon} f^*_i(l) \max_{1\leq j \leq k}(-\log(f_j(l)))\\
 & \leq   \frac{\delta}{1-\delta} + 0 + \frac{\epsilon}{2}  \leq \epsilon
  \end{split}
  \end{equation}
  for $\delta$ small enough.
  For such a $\delta$ denote
  \begin{equation*}
  \begin{split}
  \Theta_\epsilon = \{ Q ~ : ~ \lVert Q-Q^* \rVert \leq \epsilon \} \times 
  \{f_1, \dots ,f_k ~ : ~  \sum_{l>L_\epsilon} f_i^*(l)\max_{1 \leq j \leq k} (-\log(f_j(l)))<\frac{\epsilon}{2}, \\
  \forall l \leq L_\epsilon, 
  ~ \lvert f_j(l)- f^*_j(l) \rvert \leq c\delta \} 
  \end{split}
  \end{equation*}
   Using Equation (\ref{eqE3}), (A1b) holds.
  Moreover 
$$\sum_{i=1}^k \sum_{l \in \mathbb{N}} f^*_i(l) \left\lvert  \log\left( \sum_{j=1}^k f_j(l)\right) \right\rvert
\leq \sum_{i=1}^k \sum_{l\in \mathbb{N}} f^*_i(l) \left(\log(k) - \log\left(\min_{1\leq j \leq k} f_j(l)\right)\right) <+\infty $$
so that (A1e) holds.
Furthermore (A1d) and (A1c) are obviously checked.
 Using  the assumption that $Q^*$ is in the support of $\pi_Q$, (A1a) is checked. Then using Equation \eqref{eqE4},  (A1) holds and the first part of Proposition \ref{thD} follows.

We now prove the second part of Proposition \ref{thD}.
We first give a representation of a discrete Dirichlet process with independent Gamma distributed random variables.

\begin{lemma}\label{LDir}
 Let $(Z_l)_{l \in \mathbb{N}}$ be independent random variables such that 
  for all $ l \in \mathbb{N} $,
 $$Z_l \sim \Gamma(\alpha G_0(l),1),$$
 then  $\sum_{l=1}^L Z_l$ converges almost surely and its limit has a gamma distribution  $\Gamma(\alpha,1)$.
\newline
Moreover denote
$$ f : \left\{ \begin{array}{ll}
\mathbb{N} &\to [0,1]\\
i &\to f(i)=Z_i/(\sum_{l=1}^{+\infty} Z_l)
\end{array}\right.  ,$$
then $f$ is distributed from a Dirichlet process $DP(\alpha G_0)$.
\end{lemma}

\begin{proof}[Proof of Lemma \ref{LDir}]
First for all $t\in \mathbb{R}$
\begin{equation*}
\begin{split}
\lim_{L\to \infty} \mathbb{E}\left( \exp\left( i t \sum_{l=1}^L Z_l \right)\right) &
 = \lim_{L\to \infty} \prod_{l=1}^L \mathbb{E}(\exp(it Z_l)) 
 = \lim_{L\to \infty} \prod_{l=1}^L  (1-it)^{-\alpha G_0(l)} \\
& = (1-it)^{-\alpha},
\end{split}
\end{equation*}
thus $\sum_{l} Z_l$ converges in law and equivalently almost surely (see Section 9.7.1 in \citet{Du}) 
and is distributed from a gamma distribution  $\Gamma(\alpha,1)$.
\newline
Let $\{B_1, \dots, B_M\}$ be a partition of $\mathbb{N}$,
\begin{equation*}
\begin{split}
(f(B_1),\dots, f(B_M))
&= \left(\frac{\sum_{l \in B_1} Z_l}{ \sum_{l\in \mathbb{N}} Z_l}, \dots ,\frac{ \sum_{l \in B_M}Z_l}{\sum_{l \in \mathbb{N}} Z_l}\right)\\
& \sim Dir((\alpha G_0(B_1), \dots, \alpha G_0(B_M)))
\end{split}
\end{equation*}
since $\left(\sum_{l \in B_1} Z_l, \dots ,\sum_{l \in B_M}Z_l\right)$ are independent random variables and
for all $1\leq i \leq M$, $$\sum_{l \in B_i} Z_l \sim \Gamma(\alpha G_0(B_i),1). $$
Finally $f$ is drawn from a Dirichlet process $DP(\alpha G_0)$.
\end{proof}

We assume (A1b) i.e. for all $\epsilon>0$,
$$ DP(\alpha G_0)^{\otimes k }\left(\left\{f \in \mathcal{F}^k , \forall i \in \{1, \dots, k\} ~ \sum_{l \in \mathbb{N}} f^*_i(l)\max_{1\leq j \leq k} \log \frac{f_j^*(l)}{f_j(l)}<\epsilon\right\}\right)>0.$$
Let  $\epsilon>0$,
 define $\mathcal{F}_\epsilon$ as the set of $f=(f_1, \dots, f_k)$ such that for all $1\leq i \leq k$, for all $f\in \mathcal{F}_\epsilon$,
$$\sum_{l \in \mathbb{N}} f^*_i(l) \log\left( \frac{f^*_i(l)}{f_i(l)}\right)< \epsilon.$$
Then $DP(\alpha G_0)^{\otimes k }(\mathcal{F}_\epsilon)>0$.

Since $\sum_l f^*_i(l) (-\log f^*_i(l))$ converges,
then $\sum_l f^*_i(l) (-\log f_i(l))$ converges.
Using Lemma \ref{LDir}, we can write $f_i$ with independent gamma distributed random variables $(Z_l)_{l\in \mathbb{N}}$: 
$$f_i(l)=\frac{Z_l}{\sum_{j\in \mathbb{N}} Z_j},$$
where $Z_l \sim \Gamma(\alpha G_0(l),1)$.
Then $\sum_{l\in \mathbb{N}} f^*_i(l) (-\log(Z_l))$ converges since $\sum_{j \in \mathbb{N}} Z_j$ is finite almost surely.
Since $DP(\alpha G_0)^{\otimes k }(\mathcal{F}_\epsilon)>0$, for all $1\leq i \leq k$  with positive probability , 
$$\sum_{l\in \mathbb{N}} f^*_i(l) (-\log(Z_l))$$
 converges.
 Using the Kolmogorov
$0$-$1$ law and the Three-Series Theorem (see Section 9.7.3 in \citet{Du}), 
$\sum_{l\in \mathbb{N}} f^*_i(l) (-\log(Z_l))$
 converges almost surely and 
\begin{alignat}{2}
&\sum_{l\in \mathbb{N}} \mathbb{P}(\lvert f^*_i(l) (-\log(Z_l))  \rvert >1) <+ \infty \label{seq1} \\
&\sum_{l\in \mathbb{N}} \mathbb{E}\big(f^*_i(l) (-\log(Z_l))  \mathds{1}_{\lvert f^*_i(l) (-\log(Z_l)) \rvert \leq 1}\big) <+ \infty \label{seq2} \\
&\sum_{l\in \mathbb{N}} \text{var}\big(f^*_i(l) (-\log(Z_l))  \mathds{1}_{\lvert f^*_i(l) (-\log(Z_l))\rvert  \leq 1}\big) <+ \infty .
\end{alignat}
Equation \eqref{seq1} implies that
\begin{equation*}
\begin{split}
+ \infty & > \sum_{l\in \mathbb{N}} \mathbb{P}(\lvert f^*_i(l) (-\log(Z_l))  \rvert >1) \\
   &\geq \sum_{l\in \mathbb{N}} \frac{1}{\Gamma(\alpha G_0(l))}   \int_0^{\exp(-1/f^*_i(l))} x^{\alpha G_0(l)-1} e^{-x} dx \\
& \geq \sum_{l\in \mathbb{N}} \frac{1}{\alpha G_0(l) \Gamma(\alpha G_0(l))}    \exp\left(-\exp\left(\frac{-1}{f^*_i(l)}\right)  -\frac{\alpha G_0(l) }{f^*_i(l)}\right) \\
& \gtrsim  \sum_{l\in \mathbb{N}}     \exp\left(-\frac{\alpha G_0(l) }{f^*_i(l)}\right).
\end{split}
\end{equation*}
Then $$\lim_l \frac{f^*_i(l)}{G_0(l)} = 0.$$
Moreover Equation \eqref{seq2} implies that
\begin{equation*}
\begin{split}
+ \infty & > \sum_l \mathbb{E}\big(f^*_i(l) (-\log(Z_l))  \mathds{1}_{\lvert f^*_i(l) (-\log(Z_l)) \rvert\leq 1}\big) \\
   & \geq \sum_l \bigg( \int_{\exp(-1/f^*_i(l))}^{1} \frac{1}{\Gamma(\alpha G_0(l))}  f^*_i(l) (-\log(x))  x^{\alpha G_0(l)-1} e^{-x} dx  \\
     & \qquad   \qquad        + \int_{1}^{\exp(1/f^*_i(l))} \frac{1}{\Gamma(\alpha G_0(l))}  f^*_i(l) (-\log(x))  x^{\alpha G_0(l)-1} e^{-x} dx \bigg)  \\
 & \geq \sum_l  \bigg(  \frac{e^{-1} f^*_i(l)}{\Gamma(\alpha G_0(l)}  \int_{\exp(-1/f^*_i(l))}^{1}   (-\log(x))  x^{\alpha G_0(l)-1} dx \\
    & \qquad \qquad  -\frac{1}{\Gamma(\alpha G_0(l))} \int_{1}^{\exp(1/f^*_i(l))} e^{-x}dx \bigg)  \\
 &\gtrsim   - \alpha   +  \sum_l  \frac{e^{-1} f^*_i(l)}{\alpha^2 G_0^2(l)\Gamma(\alpha G_0(l))} 
           \\
& \qquad \qquad \left(  1 - \exp\left(- \frac{\alpha G_0(l) }{f^*_i(l)}\right)   -  \frac{\alpha G_0(l) }{f^*_i(l)} \exp\left(- \frac{\alpha G_0(l) }{f^*_i(l)}\right)  \right) \\
& \gtrsim  - \alpha +  \sum_l \frac{f^*_i(l)}{G_0(l)}  
\end{split}
\end{equation*}
so that $\sum_l \frac{f^*_i(l)}{G_0(l)}   < + \infty$.

%%%%%%%%%%%%%%%%%%%%%%%%%%%%%%%%%%%%%%%%%%%%%%%%%%%%%%%%%%%%%%%%%%%%%%%%%%%%%%%%%%%%%%%%%%%%%%%%%%%%%%%%%%%%
%%%%%%%%%%%%%%%%%%%%%%%%%%%%%%%%%%%%%%%%%%%%%%%%%%%%%%%%%%%%%%%%%%%%%
%%%%%%%%%%%%%%%%%%%%%%%%%%%%%%%%%%%%%%%%%%%%%%%%%%%%%%%%%%%%%%%%%%%%%%%%%%%%%%%%%%%%%%%%%%%%%%%%%%%%%%%%%%%%%
%%%%%%%%%%%%%%%%%%%%%%%%%%%%%%%%%%%%%%%%%%%%%%%%%%%%%%%%%%%%%%%%%%%%%
%%%%%%%%%%%%%%%%%%%%%%%%%%%%%%%%%%%%%%%%%%%%%%%%%%%%%%%%%%%%%%%%%%%%%%%%%%%%%%%%%%%%%%%%%%%%%%%%%%%%%%%%%%%%%%%

\section{Other proofs}\label{op}

%%%%%%%%%%%%%%%%%%%%%%%%%%%%%%%%%%%%%%%%%%%%%%%%%%%%%%%%%%%%%%%%%%%%%%%%%%%%%%%%%%%%%%%%%%%

\subsection*{Proof of Proposition \ref{th:MG}}
The proof uses many ideas of \citet{To06}.
 
We now prove that Assumptions (B1), (B2), (B3), (B4) and (B5) imply (A1). 
 A reproduction of the proof of Theorem 3.2. and Lemma 3.1 of \citet{To06} shows that Assumptions (B2), (B3), (B4) and (B5) imply that 
 for all $\epsilon>0$, for all $1 \leq j \leq k$ there exists a weak neighborhood $V_j$
 of a compactly supported probability $\tilde{P}_j$ such that for all $f_j=\phi * P_j$, $P_j\in V_j$,
 \begin{equation}\label{th51:KL}
  \int_{\mathbb{R}} f^*_i(y) \max_{1\leq j \leq k} \log\left(\frac{f^*_j(y)}{f_j(y)} \right)\lambda(dy) < \epsilon .
 \end{equation}
Let 
$0<\underline{\sigma}<\bar{\sigma}$ and $\zeta>0 $ be
such that for all $1\leq j \leq k$
$$\tilde{P}_j([-\zeta,\zeta] \times [\underline{\sigma},\bar{\sigma}])=1.$$
Let $\delta=\underline{\sigma}/2$.
 For all $1\leq j\leq k$ define
$$U_j=\{P ~ : ~ \left\lvert \int_{\mathbb{R} \times (0,+\infty)} \xi dP - 
\int_{\mathbb{R} \times(0,+\infty)} \xi d\tilde{P}_j \right\rvert <\epsilon \},$$ 
where $\xi: \mathbb{R} \times (0,+\infty) \to [0,1]$
is a piecewise affine continuous function such that 
$\xi(z,\sigma)=1$ for all $z \in [-\zeta,\zeta]$
and 
$\sigma \in   [\underline{\sigma},\bar{\sigma}]$
and $\xi(z,\sigma)=0$ for all $z \in [-\zeta-\delta,\zeta+\delta]^c$
and 
$\sigma \in   [\underline{\sigma} - \delta,\bar{\sigma} + \delta]^c$.
For all $\epsilon>0$, define  
$$\Theta_{\epsilon}=
\{ Q ~:~ \lVert Q - Q^* \rVert <\epsilon  \} \times (V_1 \cap U_1) \times \dots \times (V_k \cap U_k) .$$
Then for all $(Q,\phi*P_1, \dots, \phi*P_k) \in \Theta_\epsilon$, (A1b) is true according to Equation (\ref{th51:KL}).
In addition, for all $y\in \mathbb{R}$,
\begin{equation*}\label{th51:minor}
\begin{split}
 f_j(y)& \geq \int_{[-\zeta- \delta,\zeta+ \delta] \times [\underline{\sigma}-\delta,\bar{\sigma}+\delta]}
                         \phi_{\sigma}(y-z)P_j(dz,d\sigma)\\
       & \geq \frac{1}{\bar{\sigma} + \delta} ~ \phi_{\underline{\sigma} - \delta} 
       \big(\max(\lvert y-\xi-\delta \rvert,\lvert y+\xi+\delta \rvert)\big) ~ (1-\epsilon)
\end{split}
\end{equation*}
which implies (A1c).
Moreover using assumption (B1), $\Pi_P$-a.s. there exists $C>0$ such that for all $1\leq j \leq k$,
$$ f_j(y) \leq \int \frac{1}{\sigma} P_j(dz,d\sigma) \leq C.$$
Then 
\begin{equation*}
\begin{split}
  &\left\lvert  \log  \left( \frac{1}{k} \sum_{j=1}^k f_j(y)  \right) \right\rvert \\
 &   \leq  \lvert \log(C) \rvert  + \lvert  \log(\bar{\sigma}+\delta) \rvert
 - \log(1 - \epsilon) +  \frac{(\max( y-\xi-\delta , y+\xi+\delta ))^2}{2 (\underline{\sigma}-\delta)^2}
 \end{split}
\end{equation*}
which implies (A1e) under (B5).
Furthermore (B1) implies (A1d).
As $\Theta_\epsilon$ is a product of neighborhoods of elements in the support of their respective prior,
$\pi(\Theta_\epsilon)>0$, so (A1) is checked.

Now we prove that Assumption (B6) implies Assumption (A2).
Let $\delta>0$. For all $a,l,u, \kappa >0$, such that $l<u$ denote 
$\mathcal{F}^{\kappa}_{a,l,u}= \{\phi*P ~ : ~ P((-a,a]\times(l,u])>1- \kappa \} $.
Using Section 4 of \citet{To06}, there exist $b_0,b_1,b_2$ only depending on $\kappa$ such that
\begin{equation}
\begin{split}
 \log(N(3\kappa,(\mathcal{F}^{\kappa}_{a,l,u})^k,d))& \leq k \log(N(3\kappa,\mathcal{F}^{\kappa}_{a,l,u},\lVert \cdot \rVert_{L_1(\lambda)}))\\
 &\leq k b_0 \left(b_1 \frac{a}{l} + b_2 \log\left(\frac{u}{l} \right) +1 \right)
 \end{split}
\end{equation}
Choosing $\kappa=\frac{\delta}{3*36l} $ and $\beta<\frac{\delta^2 k \underline{q}^2}{32lb_0(b_1+b_2)}$, assumption (B6) shows that assumption (A2) holds.

%%%%%%%%%%%%%%%%%%%%%%%%%%%%%%%%%%%%%%%%%%%%%%%%%%%%%%%%%%%%%%%%%%%%%%ù

\subsection*{Proof of Corollary \ref{th3}}\label{P:th3}
 By repeating the proof of Theorem \ref{th2} and using the result of identifiability of Theorem 2.1 of \citet{GaRo13} , if $\lim_{n\to\infty} D_3(\gamma^n,\gamma^*)=0$, there 
 exists  a subsequence of $\gamma_n$, which we also denote $\gamma_n$, such that
 $Q^n$ tends to $Q^*$ and for all $1\leq j \leq k$,
 $g^n(\cdot-m^n_j)\lambda$ weakly tends to $g^*(\cdot-m^*_{j})\lambda$.
 Particularly $ g^n(\cdot)\lambda$ weakly tends to $g^*(\cdot)\lambda$.
 These weak convergences imply the pointwise convergence of the characteristic functions.
 As for all $t\in \mathbb{R}$,
 $$\int e^{ity}g^n(y-m^n_j)d\lambda(y) = e^{itm_j^n} \int e^{ity} g^n(y) d\lambda(y)$$
 then $\lim_{n\to \infty} e^{itm_j^n}= e^{itm^*_{j}}$ for all $t$ 
 such that $\int e^{ity}g^*(y) d\lambda(y) \neq 0$.
 As any characteristic function is uniformly continuous and equal to $1$ at $0$, 
 there exists $\alpha>0$ such that $\int e^{ity}g^*(y) d\lambda(y) \neq 0$ for all $\lvert t  \rvert<\alpha$. 
 Thus for all $1\leq j \leq k$, 
 $\lim_{n \to \infty} m^n_j=m^*_j$.
% Finally for all neighborhood $U_{Q^*}$ of $Q^*$, for all neighborhood $U_{0}$ of $0\in \mathbb{R}^d$ and
% for all weak neighborhood $U_{g^*}$ of $g^*\lambda$,
%there exists a $D_l$-neighborhood $U_{\gamma^*}$ of $\gamma^*$ such that
%\begin{equation*}
%\begin{split}
%U_{\gamma^*} \subset \cup_{\sigma \in \mathcal{S}_k} \sigma(U_{Q^*})\times U_0 \times 
%(U_0 +m^*_{\sigma(2)}-m^*_{\sigma(1)})\times \dots \\
%\times (U_0+m^*_{\sigma(k)}-m^*_{\sigma(1)})
%\times (\tau_{m^*_{\sigma(1)}}(U_{g^*}))
%\end{split}
%\end{equation*}
This implies the first part of Corollary \ref{th3}.

If moreover $\max_{1\leq j \leq k} \mu^*_j > \frac{1}{2}$ and $g^*$ is uniformly continuous, using the following inequality proved in the proof of Corrolary 1 in \citet{GaRo13}
\begin{equation*}
\begin{split}
\lVert D_1(\gamma^n,\gamma^*) \rVert_{L_1} 
 \geq & \left(2 \max_{1 \leq j \leq k} \mu_j^* -1\right) \lVert g^n - g^* \rVert_{L_1(\lambda)} \\
&-\max_{1\leq j \leq k} \lvert \mu^*_j - \mu^n_i \rvert
- \max_{1\leq j \leq k}\lVert g^*(\cdot-m_j^n) - g^*(\cdot-m_j^*) \rVert_{L_1(\lambda)} 
\end{split}
\end{equation*}
we obtain that $\lim _{n \to \infty} \lVert g^n - g^* \rVert_{L_1(\lambda)}=0$ which implies the last part of Corollary \ref{th3}.

%%%%%%%%%%%%%%%%%%%%%%%%%%%%%%%%%%%%%%%%%%%%%%%%%%%%%%%%%

\subsection*{Proof of Proposition \ref{th:tGM}}\label{PtGM}
As in the proof of Proposition \ref{th:MG}, many ideas come from \citet{To06}.
We first prove (A1) assuming that (B1), (B2), (B3), (B4) and (B5) are verified with $f_j(\cdot)=g(\cdot-m_j),~ 1 \leq j \leq k$. With the same ideas of the proof of Theorem 3.2 in \citet{To06}, for all $\epsilon>0$
there exists a probability $\tilde{P}$ on $\mathbb{R}\times(0,+\infty)$ 
such that there exists $0<\underline{\sigma}<\bar{\sigma}$ and $a>0$ satisfying
$$\tilde{P}((-a,a]\times(\underline{\sigma},\bar{\sigma}])=1$$
and 
$$\int g^*(y-m^*_i) \max_{1\leq j \leq k} \log \frac{g^*(y-m^*_j)}{\phi * \tilde{P} (y-m^*_j)} \lambda(dy)\leq \frac{\epsilon}{3},$$
using Assumptions (B2), (B3), (B4) and (B5).

Let $G=[-a,a]\times[\underline{\sigma},\bar{\sigma}]$.
Using the proof of Lemma 3.1 in \citet{To06} for all $C>\max_{1\leq j \leq k}\lvert m^*_j \rvert + a + \bar{\sigma}$, for all $m_j \in [m^*_j-a, m^*_j +a]$,
and for all $P$ such that $P(G)>\frac{\underline{\sigma}}{\bar{\sigma}}$,
\begin{equation}\label{eqplop}
\begin{split}
 \int_{\lvert y \rvert > C} & g^*(y-m^*_i) \max_{1 \leq j \leq k} \log \frac{\phi*\tilde{P}(y-m^*_j)}{\phi*P(y-m_j)}\lambda(dy) \\
 & \leq \int_{\lvert y \rvert > C}  g^*(y-m_i) \max_{1 \leq j \leq k} 
        \frac{1}{2} \left(\frac{\lvert y\rvert + \lvert m^*_j \rvert +2a}{\underline{\sigma}} \right)^2 \lambda(dy) <\infty
\end{split}
\end{equation}
Using assumption (B5) and Equation \eqref{eqplop}, we fix $C$ such that
$$\int_{\lvert y \rvert > C}  g^*(y-m^*_i) \max_{1 \leq j \leq k}
\log \frac{\phi*\tilde{P}(y-m^*_j)}{\phi*P(y-m_j)}\lambda(dy)\leq \frac{\epsilon}{3}$$

Let $G_\delta=[-a-\delta,a+\delta]\times[\underline{\sigma}-\delta,\bar{\sigma}+\delta]$,
with $\delta$ chosen in $(0,\min(\frac{\underline{\sigma}}{2},\frac{a}{2})]$. 
Let $\xi: \mathbb{R} \times (0,+\infty) \to [0,1]$
be a piecewise affine continuous function such that 
$\xi(z,\sigma)=1$ on $G$
and $\xi(z,\sigma)=0$ on $G_\delta^c$.
Let $$c=\inf_{\scriptsize \begin{array}{c} 
      \underline{\sigma} - \delta \leq \sigma \leq \bar{\sigma} + \delta, \\
      \lvert y \rvert \leq C, \\  
      \lvert \theta \rvert \leq a +\max_{j} \lvert m^*_j \rvert + \delta
      \end{array}}
          \phi_\sigma\left(y-\theta\right).$$
          
By Arzela-Ascoli theorem there exists $y_1, \dots, y_I$ such that for all $y \in [-C,C]$ and $1\leq j \leq k$, there exists $1 \leq i \leq I$
such that 
$$\sup_{(z, \sigma) \in G_\delta} \left\lvert \phi_\sigma\left(y-m^*_j-z \right) - \phi_\sigma\left(y_i-m^*_j-z \right)  \right\rvert < c\delta $$
Let 
\begin{equation*}
\begin{split}
V_\delta= &\bigg\{ P ~ : ~ \Big \lvert \int \xi(z,\sigma) \phi_{\sigma}(y_i-m^*_j-z)dP(z,\sigma) - \\
           & \qquad                \int \xi(z,\sigma) \phi_{\sigma}(y_i-m^*_j-z)d\tilde{P}(z,\sigma)   \Big\rvert < c\delta \bigg\}.
\end{split}                            
\end{equation*}

For all $P \in V_\delta$, for all 
$m_j \in \left[m^*_j - \frac{c \underline{\sigma} \delta \sqrt{2}}{\sqrt{\pi}},m^*_j + \frac{c \underline{\sigma} \delta \sqrt{2}}{\sqrt{\pi}}\right]$
 and for all $1\leq j \leq k$, we get
\begin{equation*}
\left\lvert \frac{\int \xi(z,\sigma) \phi_{\sigma}(y-m^*_j-z) dP(z, \sigma)}{\int \xi(z,\sigma) \phi_{\sigma}(y-m_j-z) d\tilde{P}(z, \sigma)} - 1\right\rvert
\leq 4 \delta
\end{equation*}
thus
\begin{equation*}
\begin{split}
 &\int_{\lvert y \rvert \leq C} g^*(y-m^*_i) \max_{1 \leq j \leq k} \log \frac{\phi*\tilde{P}(y-m^*_j)}{\phi*P(y-m^*_j)}\lambda(dy)\\
 &\quad \leq \int_{\lvert y \rvert \leq C} g^*(y-m^*_i) \max_{1 \leq j \leq k} 
			\log \frac{\int \xi(z,\sigma) \phi_\sigma(y-m^*_j-z)d\tilde{P}(z,\sigma)}{\int\xi(z,\sigma)\phi_\sigma(y-m^*_j-z)dP(z, \sigma)}\lambda(dy)\\
 & \quad  \leq \frac{4\delta}{1-4 \delta}			
 \end{split}
\end{equation*}

Then for $\delta$ small enough, for all $g=\phi*P$ such that $P \in V_\delta\cap \{P ~ : ~ P(G)>\frac{\underline{\sigma}}{\bar{\sigma}} \}=\tilde{V}_\delta$, for all 
$m_j \in \left[m^*_j - \frac{c \underline{\sigma} \delta \sqrt{2}}{\sqrt{\pi}},m^*_j + \frac{c \underline{\sigma} \delta \sqrt{2}}{\sqrt{\pi}}\right]=M_j^\delta$
 and for all $1\leq i \leq k$,
 \begin{equation}\label{EQ1}
\max_{1\leq i\leq k} \int g^*(y-m^*_i) \max_{1\leq j \leq k} \log \left(\frac{ g^*(y-m^*_j)}{g(y-m_j)} \right) dy < \epsilon, 
\end{equation}

moreover,
\begin{equation}\label{EQ2}
 \begin{split}
g(y-m_i) &  \geq \int_{G} \phi_\sigma(y-m_i-z)P(dz,d\sigma) \\
  &  \geq  \frac{\underline{\sigma}}{\bar{\sigma}}\phi_{\underline{\sigma}}
                      (\max(\lvert y-m_i -a \rvert, \lvert y-m_i + a \rvert))  P(G)\\
  & \geq \frac{\underline{\sigma}}{\bar{\sigma}}  \phi_{\underline{\sigma}}
                      (\max(\lvert y-m_i - a \rvert, \lvert y-m_i + a \rvert)) \frac{\underline{\sigma}}{\bar{\sigma}}>0.
 \end{split}
\end{equation}

Using assumption (B1) , there exists $\tilde{C}<0$ such that $g\leq \tilde{C}$ 
thus for all $P\in \tilde{V}_\delta$ and $m_j \in M_j^\delta$ for all $1 \leq  j\leq k$,

\begin{equation}\label{EQ3}
\begin{split}
 \sum_{i=1}^k & \mu^*_i \int g^*(y-m^*_i) \left\lvert 
 \log \left( \frac{1}{k} \sum_{j=1}^k g(y-m_j)\right) \right\rvert dy \\
 & \leq  \sum_{i=1}^k \mu^*_i \int g^*(y-m^*_i) \max_{1 \leq j\leq k} \bigg( \lvert \log(\tilde{C}) \rvert \\
  & \qquad +  2\log(\frac{\underline{\sigma}}{\bar{\sigma}}) +  \frac{(\max( y-m_j-a , y-m_j+a ))^2}{2 \underline{\sigma}^2} \bigg) \\
 & < \infty
\end{split}
 \end{equation}

Assumption (B1) ensures that (A1d) holds.

Finally for all $\epsilon>0$, there exists $\delta>0$ such that (A1) holds 
with $\Theta_\epsilon = \{ Q ~ : ~ \lVert Q -Q^* \rVert < \min(\epsilon, \underline{q}/2) \}
\times M_1^\delta \times \dots \times M_k^\delta \times \tilde{V}_\delta$ 
using Equations \eqref{EQ1}, \eqref{EQ2} and \eqref{EQ3}.

\vspace{5mm}
We now prove (C2) thanks to Assumption (D6).
Let 
$$\mathcal{F}_{a,l,u,\underline{m}} =[-\underline{m},\underline{m}]^k  \times \mathcal{F}_{a,l,u},$$
where $\mathcal{F}_{a,l,u}=\mathcal{F}_{a,l,u}^{2} $ is defined in the proof of Proposition \ref{th:MG}.
Note that for all $(m,\phi*P), (\tilde{m},\phi*\tilde{P}) \in \mathcal{F}_{a,l,u,\underline{m}}$, for all $1\leq i \leq k$,
\begin{equation*}
\begin{split}
 &\lVert \phi*P(\cdot -m_i)-\phi*\tilde{P}(\cdot - \tilde{m}_i) \rVert_{L_1(\lambda)}\\
 &\leq \lVert \phi*P(\cdot -m_i)-\phi*P(\cdot - \tilde{m}_i) \rVert_{L_1(\lambda)} 
             +\lVert \phi*P(\cdot)-\phi*\tilde{P}(\cdot) \rVert_{L_1(\lambda)}                 
 \end{split}
\end{equation*}
The second term is dealt with in the proof of Proposition \ref{th:MG}. As to the first part, we bound
\begin{equation*}
  \lVert \phi*P(\cdot -m_i)-\phi*P(\cdot - \tilde{m}_i) \rVert_{L_1(\lambda)} 
  \leq \frac{1}{l} \sqrt{\frac{2}{\pi}} \lvert m_i - \tilde{m}_i \rvert
\end{equation*}
Then for all $\kappa>0$, $a,l,u,\underline{m}>0$ such that $l<u$,
\begin{equation*}
 N(3\kappa, \mathcal{F}_{a,l,u,\underline{m}}, d)
 \leq \left(\frac{2 \underline{m}}{l\kappa}+1\right)^k  N(2\kappa, \mathcal{F}_{a,l,u}, \lVert \cdot \rVert_{L_1(\lambda)})
\end{equation*}
For all $\kappa>0$, 
let $$\mathcal{F}^\kappa_{a,l,u,\underline{m}} =[-\underline{m},\underline{m}]^k  \times \mathcal{F}^{\kappa}_{a,l,u}.$$
Following the ideas of Lemmas 4.1 and 4.2 in \citet{To06}, there exist $c_0,c_1,c_2,c_3$ only depending on $\kappa$ such that
\begin{equation*}
\log\left( N(\kappa,\mathcal{F}^\kappa_{a,l,u,\underline{m}}), d \right) 
\leq c_0 \left(c_1 k \log\frac{\underline{m}}{l} + c_2 \frac{a}{l} + c_3 \log \frac{u}{l} +1 \right),
\end{equation*}
so that (D6) implies (C2) with suitable choices of $\kappa$ and $\beta$.

\bibliographystyle{abbrvnat}
\bibliography{biblio2}

\begin{thebibliography}{18}
\providecommand{\natexlab}[1]{#1}
\providecommand{\url}[1]{\texttt{#1}}
\expandafter\ifx\csname urlstyle\endcsname\relax
  \providecommand{\doi}[1]{doi: #1}\else
  \providecommand{\doi}{doi: \begingroup \urlstyle{rm}\Url}\fi

\bibitem[Barron(1988)]{Ba88}
A.~Barron.
\newblock The exponential convergence of posterior probabilities with
  implications for bayes estimators of density functions.
\newblock Technical report, April 1988.

\bibitem[Baum and Petrie(1966)]{BaPe66}
L.~E. Baum and T.~Petrie.
\newblock Statistical inference for probabilistic functions of finite state
  markov chains.
\newblock \emph{The Annals of Mathematical Statistics}, 37\penalty0
  (6):\penalty0 1554--1563, 1966.

\bibitem[Capp\'e et~al.(2005)Capp\'e, Moulines, and Ryd\'en]{CaMoRy05}
O.~Capp\'e, E.~Moulines, and T.~Ryd\'en.
\newblock \emph{Inference in Hidden Markov Models}.
\newblock Springer, 2005.

\bibitem[de~Gunst and Shcherbakova(2008)]{GuSh08}
M.~C. de~Gunst and O.~Shcherbakova.
\newblock Asymptotic behavior of bayes estimators for hidden markov models with
  application to ion channels.
\newblock \emph{Mathematical Methods of Statistics}, 17\penalty0 (4):\penalty0
  342--356, 2008.

\bibitem[Douc and Matias(2001)]{DoMa01}
R.~Douc and C.~Matias.
\newblock Asymptotics of the maximum likelihood estimator for general hidden
  markov models.
\newblock \emph{Bernoulli}, 7:\penalty0 381--420, 2001.

\bibitem[Douc et~al.(2004)Douc, Moulines, and Ryd\'en]{Do04}
R.~Douc, E.~Moulines, and T.~Ryd\'en.
\newblock Asymptotic properties of the maximum likelihood estimator in
  autoregressive models with markov regime.
\newblock \emph{The Annals of statistics}, 32\penalty0 (5):\penalty0
  2254--2304, 2004.

\bibitem[Douc et~al.(2011)Douc, Moulines, Olsson, and van Handel]{DoMoOlHa11}
R.~Douc, E.~Moulines, J.~Olsson, and R.~van Handel.
\newblock Consistency of the maximum likelihood estimator for general hidden
  {M}arkov models.
\newblock \emph{The Annals of Statistics}, 39\penalty0 (1):\penalty0 474--513,
  2011.

\bibitem[Dudley(2002)]{Du}
R.~M. Dudley.
\newblock \emph{Real analysis and probability}, volume~74.
\newblock Cambridge University Press, 2002.

\bibitem[Dumont and Le~Corff(2012)]{DuCo12}
T.~Dumont and S.~Le~Corff.
\newblock Nonparametric regression on hidden phi-mixing variables:
  identifiability and consistency of a pseudo-likelihood based estimation
  procedure.
\newblock \emph{arxiv preprint arXiv:1209.0633}, 2012.

\bibitem[Gassiat and Rousseau(2013{\natexlab{a}})]{GaRo12}
E.~Gassiat and J.~Rousseau.
\newblock About the posterior distribution in hidden markov models with unknown
  number of states.
\newblock \emph{Bernoulli}, 2013{\natexlab{a}}.
\newblock to appear.

\bibitem[Gassiat and Rousseau(2013{\natexlab{b}})]{GaRo13}
E.~Gassiat and J.~Rousseau.
\newblock Non parametric finite translation hidden markov models and
  extensions.
\newblock \emph{Bernoulli}, 2013{\natexlab{b}}.
\newblock to appear.

\bibitem[Gassiat et~al.(2013)Gassiat, Cleynen, and Robin]{GaClRo13}
E.~Gassiat, A.~Cleynen, and S.~Robin.
\newblock Finite state space non parametric hidden markov models are in general
  identifiable.
\newblock \emph{arXiv preprint arXiv:1306.4657}, 2013.

\bibitem[Ghosh and Ramamoorthi(2003)]{GhRa03}
J.~Ghosh and R.~Ramamoorthi.
\newblock \emph{Bayesian Nonparametrics}.
\newblock Springer, 2003.

\bibitem[MacDonald and Zucchini(1997)]{macdonald:zucchini:1997}
I.~L. MacDonald and W.~Zucchini.
\newblock \emph{Hidden {M}arkov and other models for discrete-valued time
  series}.
\newblock Chapman and Hall/CRC, London, UK, 1997.

\bibitem[MacDonald and Zucchini(2009)]{macdonald:zucchini:2009}
I.~L. MacDonald and W.~Zucchini.
\newblock \emph{Hidden {M}arkov models for time series: an introduction using
  {R}}.
\newblock Chapman and Hall/CRC, London, UK, 2009.

\bibitem[Rio(2000)]{Ri00}
E.~Rio.
\newblock In{\'e}galit{\'e}s de hoeffding pour les fonctions lipschitziennes de
  suites d{\'e}pendantes.
\newblock \emph{Comptes Rendus de l'Acad{\'e}mie des Sciences-Series
  I-Mathematics}, 330\penalty0 (10):\penalty0 905--908, 2000.

\bibitem[Tokdar(2006)]{To06}
S.~T. Tokdar.
\newblock Posterior consistency of {D}irichlet location-scale mixture of
  normals in density estimation and regression.
\newblock \emph{Sankhy\=a}, 68\penalty0 (1):\penalty0 90--110, 2006.

\bibitem[Yau et~al.(2011)Yau, Papaspiliopoulos, Roberts, and
  Holmes]{YaPaRoHo11}
C.~Yau, O.~Papaspiliopoulos, G.~Roberts, and C.~Holmes.
\newblock Bayesian non-parametric hidden markov models with applications in
  genomics.
\newblock \emph{Journal of the Royal Statistical Society}, 73:\penalty0 37--57,
  2011.

\end{thebibliography}

\end{document}